\documentclass[nonatbib]{elsarticle}

\makeatletter
\let\c@author\relax
\makeatother

\usepackage[margin=1in]{geometry}
\usepackage{amssymb,amsmath}
\usepackage{amsthm}
\usepackage{bm}
\usepackage{setspace}
\usepackage{enumitem}
\usepackage{xcolor}
\usepackage{graphicx}
\usepackage{url}
\usepackage{booktabs}
\usepackage{csvsimple}
\usepackage{multicol,multirow}
\usepackage{subcaption}
\usepackage[export]{adjustbox}
\usepackage[scaled=0.9]{newpxtext}
\usepackage[scaled=0.9]{newpxmath}

\usepackage{hyperref}
\hypersetup{
  colorlinks =true,
  linkcolor = blue,
  filecolor = blue,
  citecolor = magenta,  
  urlcolor = magenta,
  }

\theoremstyle{remark}

\newtheorem{definition}{Definition}

\newtheorem{proposition}{Proposition}

\newtheorem{corollary}{Corollary}
\usepackage[
backend=biber,
style=authoryear,
maxcitenames=2
]{biblatex}
\usepackage{xstring}

\newcommand{\dashToTO}[1]{%
  \IfBeginWith{#1}{TO}{%
    TO ($<1\%$)%
  }{%
    \StrSubstitute{#1}{-}{TO}[\temp]%
    \temp
  }%
}

\allowdisplaybreaks 
\journal{JORS}
\begin{document}

\begin{frontmatter}
\title{Equitable Routing--Rethinking the Multiple Traveling Salesman Problem}

\author[1]{Abhay Singh Bhadoriya}
\ead{abhaysb@amazon.com}
\affiliation[1]{organization={Amazon}, 
city={Seattle}, state={Washington}, postcode={98109}, country={USA}}

\author[2]{Deepjyoti Deka}
\ead{deepj87@mit.edu}
\affiliation[2]{organization={MIT Energy Initiative}, 
city={Cambridge}, state={Massachusetts}, postcode={02142}, country={USA}}

\author[3]{Kaarthik Sundar\corref{cor}}
\ead{kaarthik@lanl.gov}
\affiliation[3]{organization={Los Alamos National Laboratory}, 
city={Los Alamos}, state={New Mexico}, postcode={87545}, country={USA}}
\cortext[cor]{Corresponding author}

\begin{abstract}
The Multiple Traveling Salesman Problem (MTSP) extends the traveling salesman problem by assigning multiple salesmen to visit a set of targets from a common depot, with each target visited exactly once while minimizing total tour length. A common variant, the min-max MTSP, focuses on workload balance by minimizing the longest tour, but it is difficult to solve optimally due to weak linear relaxation bounds. This paper introduces two new parametric fairness-driven variants of the MTSP: the $\varepsilon$-Fair-MTSP and the $\Delta$-Fair-MTSP, which promote equitable distribution of tour lengths while controlling overall cost. The $\varepsilon$-Fair-MTSP is formulated as a mixed-integer second-order cone program, while the $\Delta$-Fair-MTSP is modeled as a mixed-integer linear program. We develop algorithms that guarantee global optimality for both formulations. Computational experiments on benchmark instances and real-world applications, including electric vehicle fleet routing, demonstrate their effectiveness. Furthermore, we show that the algorithms presented for the fairness-constrained MTSP variants can be used to obtain the Pareto front of a bi-objective optimization problem in which one objective minimizes the total tour length and the other balances the lengths of the individual tours. Overall, these fairness-constrained MTSP variants provide a practical and flexible alternative to the min-max MTSP.
\end{abstract}




\end{frontmatter}

\section*{Practitioner summary} Workload balancing is an essential aspect of many decision-making problems in practical applications beyond MTSP, such as scheduling, facility location, supply chain management, and inventory management, and most existing approaches to enforce workload balancing involve changing the optimization objective to a min-max formulation. The fair-MTSP formulated in this paper and the algorithms to solve it can be extended to other applications in which some notion of fairness or workload balancing needs to be incorporated, with minimal changes to both the problem formulation and the algorithms to solve it to optimality. Furthermore, it will enable practitioners to quantify the inherent trade-off between maximizing efficiency and enforcing fairness. As for the MTSP itself, the results of this paper can be applied immediately to managing a fleet of electric vehicles or small drones for package delivery, surveillance, and capacitated vehicle routing problems in logistics, among other applications. 

\section{Introduction}
The single-depot Multiple Traveling Salesman Problem (MTSP) is a generalization of the well-known Traveling Salesman Problem (TSP). Given a depot, a set of targets, and multiple salesmen stationed at the depot, the objective is to find, for each salesman, a tour through a subset of targets such that each tour starts and ends at the depot and every target is visited exactly once. The goal is to minimize the total length of all tours. The MTSP has numerous applications in both civilian and military domains, including transportation and package delivery (see \cite{VRPDroneDelivery2016,sundar2022branch,murray2020multiple}), multi-robot task allocation (see \cite{bektas2006multiple}), disaster recovery, and surveillance (see \cite{DubinsTSP2008,RathinamDubinsTSP2007,sundar2015exact,venkatachalam2018two}), among others. It is also a special case of the broader class of Vehicle Routing Problems (VRP) (see \cite{otto2018optimization}), which introduces additional constraints related to vehicle capacities and target constraints. When these capacity constraints are ignored, the VRP reduces to the MTSP. 

In its standard form, the MTSP has a ``min-sum'' or ``min-cost'' objective that seeks to minimize the sum of the individual tour lengths. However, this formulation often leads to an imbalance in tour lengths among salesmen (\cite{bertsimas2011price}). For applications such as package delivery, persistent surveillance, and electric vehicle fleet management (\cite{francca1995m,scott2020market,hari2020optimal}), a more equitable distribution of tour lengths is desirable. Intuitively, enforcing fairness in tour lengths comes at the cost of efficiency, creating an inherent trade-off between minimizing the sum of tour lengths and ensuring equitable distribution. Various approaches have been proposed to achieve fairness in the MTSP by introducing equity-based objective functions. 

The most widely used fairness-driven objective function in the literature is the min-max objective, which seeks to minimize the length of the longest tour (\cite{applegate2002solution}). We refer to this variant as the ``min-max MTSP''. Solving both the min-sum and min-max versions of the MTSP optimally typically involves formulating them as Mixed Integer Linear Programs (MILPs) and applying state-of-the-art branch-and-cut techniques (\cite{bektas2006multiple,applegate2002solution,francca1995m,sundar2015exact}). However, the min-max MTSP has received less attention than its min-sum counterpart because its continuous relaxations yield weak lower bounds, which in turn lead to poor performance of branch-and-cut algorithms.

\textcolor{black}{A more recent equity-based objective function is the finite-dimensional $\ell^p$ norm of the vector of tour lengths (where $p \in \mathbb{Z}^{+}$). Notably, when $p=1$, the $\ell^p$ norm reduces to the classical min-sum MTSP formulation, while as $p \to \infty$, it recovers the min-max MTSP. Thus, the $p$-norm objective provides a continuum between efficiency (min-sum) and fairness (min-max). It has been empirically observed in \cite{bektacs2020using} that using the $\ell^p$ norm objective function with $p=2$ promotes fairness in tour length distribution. However, the appropriate choice of $p$ for optimizing fairness remains unclear. Increasing $p$ can enhance or reduce fairness in tour lengths; a trend also observed in \cite{sundar2024parametric} for power grid optimization problems. This uncertainty makes selecting the right $p$ to achieve a given efficiency-fairness trade-off value challenging. We refer to this variant as the $p$-norm MTSP.}

For both the $p$-norm and the min-max versions of the MTSP, the fairness of the distribution of
tour lengths is measured \textit{a posteriori} using fairness indices or metrics. The most commonly used metrics are the Gini coefficient (\cite{Gini1936}) and the Jain et al. index (\cite{Jain1984}), with the former originating in economics and the latter in communication networks.
If $\bm{l}$ denotes the $m$-dimensional non-negative vector of feasible tour lengths for the MTSP, then the Gini coefficient is defined as:
\begin{gather}
    \text{Gini Coefficient: } \mathrm{GC}(\bm l) \triangleq \frac{\sum_{1 \leqslant i \leqslant j \leqslant m} | l_i - l_j |}{(m-1) \cdot \sum_{i=1}^m l_i}. \label{eq:gi}
\end{gather}
The Gini coefficient $\mathrm{GC}$ takes a value of $1$ when the vector is most unfair (i.e., exactly one element is non-zero) and a value of $0$ when all elements are equal. Since $\bm{l}$ is a non-negative vector, $\mathrm{GC}$ is always bounded within $[0, 1]$. \textcolor{black}{We note that the definition of the Gini coefficient in Eq. \eqref{eq:gi} differs from some standard formulations in the literature by a scaling factor. This normalization is chosen so that $\mathrm{GC}(\bm l) \in [0,1]$ for any non-negative vector $\bm l$, which is convenient for directly parameterizing fairness levels. Since all such definitions are equivalent up to scaling, this choice does not affect any of the results presented in the article.}
Similarly, the Jain et al. index for $\bm{l}$ is given by:
\begin{gather}
    \text{Jain et al. Index: } \mathrm{JI}(\bm l) \triangleq \frac{1}{m} \cdot \frac{\left( \sum_{i=1}^{m} l_i \right)^2}{\sum_{i=1}^{m} l_i^2}. \label{eq:ji}
\end{gather}
The Jain index $\mathrm{JI}$ takes a value of $1/m$ when exactly one element is non-zero and $1$ when all elements are equal. Like the Gini coefficient, $\mathrm{JI}$ is bounded, with its range being $[1/m, 1]$. Throughout this article, we use these metrics to evaluate fairness in the tour-length distributions across different MTSP formulations.

For a broader discussion on fairness in combinatorial optimization problems, we refer readers to \cite{bektacs2020using}. Here, we review the literature on enforcing fairness or workload balancing in MTSP and VRP variants not covered in the previous section. Due to the challenges of developing exact approaches for the min-max MTSP and its VRP variants--primarily stemming from weak continuous relaxations--\cite{zukerman2008fair} introduced the concept of $(\alpha, \beta)$-fairness, which enforces fairness by imposing upper and lower bounds on tour lengths. Later, \cite{bektacs2013balancing} formulated the MTSP with tour-length constraints using flow-based MILP constraints and solved it optimally via Bender's Decomposition, reporting results for instances with up to 70 nodes. Another approach to fairness in MTSP and VRP is bi-objective optimization, which balances the min-sum and min-max objectives. For example, \cite{sarpong2013column} developed a column generation method to compute lower bounds for a bi-objective VRP, while \cite{halvorsen2016bi} explored different route balancing criteria and their impact on Pareto-optimal frontiers.
Additionally, \cite{matl2018workload} provides a comprehensive survey on workload balancing in VRP variants, though it does not focus on exact algorithmic solutions. More recently, \cite{kinable2017exact} introduced the equitable TSP, which seeks to construct two perfect matchings in a network that together form a tour while minimizing the cost difference between them. This work contributes to fairness-driven routing by explicitly accounting for equity in cost distribution. In summary, existing approaches to workload balancing in MTSP typically fall into three categories: (i) modifying the objective function to integrate fairness, as seen in the min-max MTSP and $p$-norm MTSP, (ii) introducing explicit lower and upper bounds on tour lengths to enforce balance, or (iii) resorting to bi-objective optimization to include both the efficiency and fairness objectives. This paper takes a fundamentally different approach to incorporating fairness by leveraging recent advancements in \cite{sundar2024parametric} and enforcing fairness as a parameterized constraint, where the parameter quantifies the level of fairness enforced. This approach enables the development of computationally efficient exact algorithms to solve the resulting fair variants of the MTSP, which allows a systematic exploration of the efficiency-fairness trade-off. We also show that our constraint-driven fairness approach is equivalent to scalarizing (see \cite{hwang2012multiple}) a bi-objective MTSP with the min-sum and fairness-maximization objectives. 

These approaches highlight the challenge of balancing fairness and efficiency in MTSP and VRP formulations, demonstrating the complexity of these formulations and the diversity of methods used to address them. Beyond MTSP, workload balancing is critical in scheduling, facility location, and other decision-making problems. Typically, fairness is enforced using a min-max objective, similar to the min-max MTSP. The fairness-driven MTSP formulations and solution approaches presented here can be adapted to other combinatorial optimization problems in which fairness is essential.
 
\subsection{Challenges in existing fairness-oriented MTSP} \label{subsec:challenges}
Both the min-max MTSP and the $p$-norm MTSP have two significant drawbacks. Firstly, both variants can be notoriously difficult to solve to optimality; the former because of the weakness in the bounds provided by its continuous relaxation and the latter because they are formulated as Mixed-Integer Convex Programs (MICPs) \cite{bektacs2020using}. MICPs cannot be solved optimally without algorithmic enhancements when using off-the-shelf solvers for $p > 2$. Secondly, both variants yield feasible MTSP solutions whose distribution of tour lengths is fair, but do not provide a continuum of solutions with varying fairness levels. The only way to explore different fairness-cost trade-offs in the $p$-norm MTSP is through brute-force methods: solving for multiple $p$ values, computing fairness indices, and comparing results.

\subsection{Contributions} \label{subsec:contributions}
To overcome the limitations of the min--max and $p$-norm variants of the MTSP, this article introduces two new parameterized formulations: the $\varepsilon$-Fair-MTSP ($\varepsilon$-F-MTSP) and the $\Delta$-Fair-MTSP ($\Delta$-F-MTSP). For brevity, we use the term F-MTSP to refer to either variant, with the intended meaning made clear from context.
These formulations enforce an equitable distribution of tour lengths among salesmen as a Second-Order Cone (SOC) constraint, or as a linear constraint, respectively, while still minimizing the total tour length.

$\varepsilon$-F-MTSP incorporates fairness into the MTSP formulation using a recently introduced model called \emph{$\varepsilon$-fairness} in \cite{sundar2024parametric}, which is enforced through a parametric SOC constraint. The SOC constraint includes a parameter $\varepsilon \in [0, 1]$, where $\varepsilon = 0$ imposes no fairness constraints, and $\varepsilon = 1$ enforces complete equality among tour lengths. Intermediate values of $\varepsilon \in (0, 1)$ correspond to varying levels of fairness. Since $\varepsilon$-fairness is represented as an SOC constraint, this variant, referred to as $\varepsilon$-F-MTSP, results in a Mixed-Integer Second-Order Cone Program (MISOCP). We establish key theoretical properties of this MISOCP and develop computationally efficient algorithms to solve it optimally. Furthermore, we demonstrate that these algorithmic approaches can also be leveraged to efficiently solve the $p$-norm MTSP for any $p \in \mathbb{Z}^{+}$. The $\Delta$-Fair-MTSP enforces fairness using a single linear constraint that bounds the Gini coefficient $\mathrm{GC}$ by an upper limit $\Delta \in [0, 1]$. This formulation ensures that fairness is explicitly controlled through $\Delta$, and we show that it can be solved optimally using existing state-of-the-art techniques. \textcolor{black}{We clarify that the proposed $\varepsilon$-F-MTSP and $\Delta$-F-MTSP formulations are not direct substitutes for the min-max MTSP. The min-max MTSP explicitly minimizes the maximum tour length and is therefore the appropriate model when controlling the worst-case workload is the primary objective. In contrast, our formulations minimize the total tour length subject to a parametric fairness constraint that regulates the dispersion of tour lengths. As such, they do not directly optimize the maximum tour length. Rather, they provide a flexible framework to explore the trade-off between efficiency and fairness. Notably, for sufficiently strict fairness levels (i.e., large $\varepsilon$ or small $\Delta$), the resulting solutions tend to approximate min-max behavior by tightly controlling variability in tour lengths, while also enabling intermediate solutions that may achieve comparable balance at lower total cost and often with improved computational performance. Accordingly, these formulations should be viewed as complementary tools for navigating the fairness--efficiency trade-off, rather than exact surrogates for the min-max MTSP.}

In addition, we formulate bi-objective versions of both $\varepsilon$-F-MTSP and $\Delta$-F-MTSP, considering a ``min-sum'' objective alongside a ``fairness-maximization'' objective. We establish that $\varepsilon$-F-MTSP and $\Delta$-F-MTSP can be interpreted as scalarizations of these bi-objective formulations (see \cite{hwang2012multiple}). The scalarizations enable the development of an efficient algorithm to compute the Pareto front, capturing the trade-off between fairness and efficiency. To the best of our knowledge, this is the first work to achieve this. Finally, through extensive computational experiments, we (i) validate the effectiveness of our algorithms in solving both $\varepsilon$-F-MTSP and $\Delta$-F-MTSP as well as the $p$-norm MTSP, (ii) quantify the trade-off between cost and fairness using the proposed $\varepsilon$-F-MTSP and $\Delta$-F-MTSP formulations, and (iii) demonstrate the superiority of $\varepsilon$-F-MTSP and $\Delta$-F-MTSP over the $p$-norm MTSP and min-max MTSP in terms of fairness enforcement and computational efficiency.

\subsection{Paper organization} \label{subsec:org}
The rest of the article is structured as follows: Section \ref{sec:mtsp} we present some mathematical preliminaries and formulate all the variants of the MTSP, Section \ref{sec:properties} presents some theoretical properties of the two $\varepsilon$-F-MTSP and $\Delta$-F-MTSP, Section \ref{sec:algo} develops algorithms to solve all the variants of the MTSP presented in the paper, Section \ref{sec:results} presents results of extensive computational experiments to corroborate the effectiveness of the formulations and algorithms followed by conclusion and way forward in Section \ref{sec:conclusion}. 

\section{Mathematical preliminaries and formulations} \label{sec:mtsp}
We start by introducing the notations. To that end, we let $T$ denote the set of $n$ targets (i.e., customers) and $d$ denote the depot where $m$ salesmen are stationed. The MTSP is then formulated on an undirected graph $G = (V, E)$ where $V \triangleq T \bigcup \{d\}$ is the vertex set and $E \triangleq \{(i, j): i, j \in V, i\neq j\}$ is the edge set. Since the graph is undirected, $(i, j)$ and $(j, i)$ refer to the same edge that connects vertices $i$ and $j$. Associated with each edge $(i, j) \in E$ is a non-negative value $c_{ij}$ that denotes the cost of the edge. Depending on the application, this cost can be used to model distance, time, or anything else. We also assume that the costs satisfy the triangle inequality, i.e., for every $i, j, k \in V$, and $c_{ij} \leqslant c_{ik} + c_{ki}$. Finally, given a subset of vertices $S \subset V$, we define 
\begin{gather*}
    \delta(S) \triangleq \{(i, j) \in E: i \in S, j \notin S \} \text{ and } \\ 
    \gamma(S) \triangleq \{(i, j) \in E: i, j \in S\}. 
\end{gather*}
Furthermore, if $S = \{i\}$, we use $\delta(i)$ in place of $\delta(\{i\})$. 

\subsection{Min-sum multiple traveling salesman problem} \label{subsec:mtsp-min-sum}
Using the notation introduced in the previous paragraph, we formulate the single-depot MTSP as an MILP, inspired by existing formulations for vehicle routing problems (see \cite{toth2002vehicle}). Associated with every salesman $v$ and edge $(i, j) \in E$ is an integer decision variable $x_{ij}^v$ whose value is the number of times the edge $(i, j)$ is in the salesman $v$'s tour. Note that for some edges $(i, j) \in E$, $x_{ij}^v \in \{0, 1, 2\}$, i.e., we permit the degenerate case where a tour for salesman $v$ can consist of just the depot and a target. If $e \in E$ connects two vertices $i$ and $j$, then $(i, j)$ and $e$ will be used interchangeably to denote the same edge. Also associated with each salesman $v$ and each vertex $i \in T$ is a binary decision variable $y_i^v$ that takes the value $1$ if $i$ is visited by salesman $v$ and $0$ otherwise. Finally, for every salesman $v$, we let $l_v$ denote the length of their tour. Using the above notations, the min-sum MTSP can be formulated as follows:
\begin{subequations}
\begin{flalign}
    & (\mathcal F_1): \quad \min ~~ \sum_{1 \leqslant v \leqslant m} l_v \quad \text{ subject to:} & \label{eq:mtsp-obj} \\
    & \sum_{e \in E} c_e x_e^v = l_v \quad \forall v \in \{1, \cdots, m\} \label{eq:mtsp-l} \\
    & \sum_{e \in \delta(i)} x^v_e = 2 \cdot y_i^v \quad \forall i \in T, v \in \{1, \cdots, m\}  \label{eq:mtsp-degree} \\
    & \sum_{e \in \delta(S)} x^v_e \geqslant 2 \cdot y_i^v \quad \forall i \in S, S\subseteq T, v \in \{1, \cdots, m\} \label{eq:mtsp-sec} \\
    & \sum_{1 \leqslant v \leqslant m} y_i^v = 1 \quad \forall i \in T \label{eq:mtsp-target-visit} \\ 
    & y_d^v = 1 \quad \forall v \in \{1, \cdots, m\} \label{eq:mtsp-depot-visit} \\
    & x_e^v \in \{0, 1, 2\} \quad \forall e \in \{(d, i): i \in T\}, v \in \{1, \cdots, m\} \label{eq:mtsp-edge-depot}\\ 
    & x_e^v \in \{0, 1\} \quad \forall e \in \{(i, j): i, j \in T\}, v \in \{1, \cdots, m\} \label{eq:mtsp-edge-target}\\ 
    & y_i^v \in \{0, 1\} \quad \forall i \in V, v \in \{1, \cdots, m\} \label{eq:mtsp-vertex} \\ 
    & l_v \geqslant 0 \quad \forall v \in \{1, \cdots, m\} \label{eq:mtsp-l-ge0}
\end{flalign}
\label{eq:mtsp}
\end{subequations}
The objective \eqref{eq:mtsp-obj} minimizes the sum of all the tour lengths; the tour lengths are defined by the constraint \eqref{eq:mtsp-l}. Constraints \eqref{eq:mtsp-degree} and \eqref{eq:mtsp-sec} are the degree and sub-tour elimination constraints, respectively, that together enforce the feasible solution for every salesman to be a tour starting and ending at the depot. \eqref{eq:mtsp-target-visit} and \eqref{eq:mtsp-depot-visit} ensure that all the salesmen start and end at the depot and exactly one salesman visits every target $i \in T$. Finally, \eqref{eq:mtsp-edge-depot}--\eqref{eq:mtsp-l-ge0} enforce the integrality and non-negativity restrictions on the decision variables. 

\subsection{$p$-norm multiple traveling salesman problem} \label{subsec:mtsp-min-p-norm}
We now formulate the $p$-norm MTSP that seeks to enforce equitability in the distribution of the tour lengths using the $\ell^p$ norms with $p \geqslant 1$. The formulations presented in this section are based on existing work in \cite{bektacs2020using}. We start by introducing some additional notations. Given that $\bm l \in \mathbb R_{\geqslant 0}^m$ denotes the non-negative vector of tour lengths for the $m$ salesman, we define 
\begin{gather}
    \| \bm l \|_p \triangleq \left( \sum_{1 \leqslant i \leqslant m} l_i^p \right)^{\frac 1p}. \label{eq:p-norm}
\end{gather}
When $p = \infty$, the definition of $\ell^{\infty}$ norm is 
\begin{gather}
 \|\bm l \|_{\infty} \triangleq \max_{1 \leqslant i \leqslant m} l_i \label{eq:linfnorm}.
\end{gather}
For ease of exposition, we shall restrict $p$ to be an integer, although our presentation also holds any real $p$. Also, it is easy to see that for any $p \geqslant 1$, the $\ell^p$ norm is a convex function. Given these notations, the $p$-norm MTSP for $p \in [1, \infty)$ is formulated as follows:
\begin{flalign}
    & (\mathcal F_p): \quad \min ~~ \left(\sum_{1 \leqslant v \leqslant m} l_v^p\right)^{\frac 1p} \quad \text{ subject to: \eqref{eq:mtsp-l} -- \eqref{eq:mtsp-l-ge0}.} \label{eq:p-mtsp} 
\end{flalign}
Notice that when $p = 1$, the formulation $\mathcal F_p$ is equivalent to the MTSP formulation \textcolor{black}{$\mathcal F_1$} in \eqref{eq:mtsp}. Also, for any value of $p \in (1, \infty)$, $\mathcal F_p$ is a mixed-integer convex optimization problem due to the presence of $l_v^p$ in the objective function in \eqref{eq:p-mtsp}. In the later sections, we develop specialized algorithms to solve the $p$-norm MTSP for any value of $p \in [1, \infty)$ to global optimality. The formulation for $p = \infty$ is provided in the next section. 

\subsection{Min-max multiple traveling salesman problem} \label{subsec:mtsp-min-max}
The mathematical formulation of the min-max MTSP is equivalent to the $p$-norm MTSP when $p = \infty$. This equivalence also justifies the current use of min-max objective functions to enforce fairness in the distribution of tour lengths or, in general, workload balancing for optimization problems in other applications. The formulation for the min-max MTSP is given by 
\begin{flalign}
    & (\mathcal F_{\infty}): \quad \min \left( \max_{1 \leqslant v \leqslant m}  l_v \right) \quad \text{ subject to: \eqref{eq:mtsp-l} -- \eqref{eq:mtsp-l-ge0}.} \label{eq:min-max-mtsp} 
\end{flalign}
The above formulation can be converted to a MILP by the introduction of an additional auxiliary variable $z \geqslant 0$ as follows:
\begin{subequations}
    \begin{flalign}
     & (\mathcal F_{\infty}): \quad \min z \quad \text{ subject to: \eqref{eq:mtsp-l} -- \eqref{eq:mtsp-l-ge0}} \label{eq:min-max-mtsp-obj} \\ 
     & z \geqslant l_v \quad \forall v \in \{1, \cdots, m\} \text{ and } z \geqslant 0 \label{eq:min-max-mtsp-aux}
\end{flalign}
\label{eq:min-max-mtsp-lin}
\end{subequations}
Next, we present two novel formulations for the F-MTSP to enforce different levels of fairness in the distribution of tour lengths.

\subsection{Fair multiple traveling salesman problem}  \label{subsec:f-mtsp}
We propose mathematical formulations for two variants of the F-MTSP. These variants differ in the manner they enforce the fairness in the distribution of tour lengths. The first variant, referred to as $\varepsilon$-F-MTSP, is based on a recently developed model for fairness called $\varepsilon$-fairness in \cite{sundar2024parametric}; $\varepsilon = 0$ corresponds to no fairness constraints being enforced on the F-MTSP making it equivalent to the min-sum MTSP, and $\varepsilon = 1$ corresponds to enforcing all the tour lengths to be equal. $\varepsilon$-F-MTSP will result in a MISOCP for a fixed value of $\varepsilon > 0$. The second variant, referred to as $\Delta$-F-MTSP, will enforce an upper-bound of $\Delta \in [0, 1]$ on the Gini coefficient of the tour lengths \eqref{eq:gi} using a single linear constraint. $\Delta = 1$  corresponds to a trivial upper bound on the Gini coefficient and is equivalent to the min-sum MTSP and as $\Delta$ is decreased, the level of fairness enforced increases and $\Delta = 0$ corresponds to enforcing all tour lengths to be equal. We start by presenting the formulation for $\varepsilon$-F-MTSP. 

\textcolor{black}{We note that $\varepsilon = 1$ (respectively, $\Delta = 0$) corresponds to enforcing equality of all tour lengths in principle. However, such requirements may render the problem infeasible for a given instance. In general, the feasibility of the $\varepsilon$-F-MTSP and $\Delta$-F-MTSP depends on the chosen values of $\varepsilon$ and $\Delta$, and there exist instance-dependent ranges of admissible values for which feasible solutions exist. Enforcing overly strict fairness (i.e., large $\varepsilon$ or small $\Delta$) can therefore lead to infeasibility. These feasibility domains and their properties are formally characterized in Section \ref{sec:properties}.}

\subsubsection{$\varepsilon$-fair multiple traveling salesman problem} \label{subsubsec:eps-f-mtsp}
We begin by invoking a well-known inequality that relates the $\ell_1$-norm and $\ell_2$-norm of the tour length vector $\bm{l} \in \mathbb{R}_{\geqslant 0}^m$, derived from the fundamental result on the \textit{equivalence of norms} (see \cite{horn2012matrix}):
\begin{gather}
 \|\bm{l} \|_2 \leqslant \| \bm{l} \|_1 \leqslant \sqrt{m} \cdot \|\bm{l}\|_2. \label{eq:norm-eq}
\end{gather} 
In \eqref{eq:norm-eq}, we observe that:

\begin{enumerate}[label=(\roman*)]
 \item\label{it:unfair} The first inequality holds with equality, i.e., $\|\bm{l} \|_2 = \| \bm{l} \|_1$, when all but one component of $\bm{l}$ are zero (\textit{most unfair case}).
 \item\label{it:fair} The second inequality holds with equality, i.e., $\| \bm{l}\|_1 = \sqrt{m} \cdot \|\bm{l}\|_2$, when all components of $\bm{l}$ are equal (\textit{most fair case}).
\end{enumerate}
Using the intuition provided by \ref{it:unfair} and \ref{it:fair}, we introduce the following parameterized definition of fairness:
\begin{definition} \label{def:eps-fairness} 
\textbf{$\varepsilon$-Fairness} (from \cite{sundar2024parametric}) — For any $\bm{l} \in \mathbb{R}_{\geqslant 0}^m$ and $\varepsilon \in [0, 1]$, we say $\bm{l}$ is \textit{$\varepsilon$-fair} if 
\begin{gather}
    \left(1-\varepsilon + \varepsilon\sqrt{m} \right) \cdot \|\bm{l} \|_2 = \|\bm{l}\|_1,
    \quad \text{or equivalently,} \quad
    \varepsilon = \frac{1}{\sqrt{m}-1} \left(\frac{\|\bm{l}\|_1}{\|\bm{l}\|_2} -1 \right). \label{eq:e-fair-eq}
\end{gather}
Here, we interpret $\varepsilon = 0$ as the \textit{most unfair} case and $\varepsilon = 1$ as the \textit{most fair} case. Furthermore, we say $\bm{l}$ is \textit{at least $\varepsilon$-fair} if 
\begin{gather}
 \left(1-\varepsilon + \varepsilon\sqrt{m} \right) \cdot \|\bm{l} \|_2 \leqslant \|\bm{l}\|_1,
 \quad \text{or equivalently,} \quad
 \varepsilon \leqslant \frac{1}{\sqrt{m}-1} \left(\frac{\|\bm{l}\|_1}{\|\bm{l}\|_2} -1 \right). \label{eq:soc-e-fair}
\end{gather}
\end{definition}
We note that the term $\left(1-\varepsilon + \varepsilon\sqrt{m} \right)$ in Definition \ref{def:eps-fairness} is a convex combination of $\sqrt{m}$ and $1$, where $\varepsilon$ acts as the weighting factor. Additionally, \eqref{eq:e-fair-eq} naturally reduces to \ref{it:unfair} when $\varepsilon = 0$ and to \ref{it:fair} when $\varepsilon = 1$. The inequality in \eqref{eq:soc-e-fair} also represents a second-order cone (SOC) constraint for any fixed $\varepsilon \in [0,1]$.
Definition \ref{def:eps-fairness} further allows us to define a fairness metric, which we refer to as the $\varepsilon$-fair index:
\begin{gather}
    \text{$\varepsilon$-Fair Index: } \mathrm{\varepsilon FI}(\bm{l}) \triangleq \frac{1}{\sqrt{m}-1} \left(\frac{\|\bm{l}\|_1}{\|\bm{l}\|_2} -1 \right). \label{eq:efi} 
\end{gather}
The range of $\mathrm{\varepsilon FI}(\bm{l})$ is $[0,1]$. We later use \eqref{eq:efi} in the context of an equivalent bi-objective formulation for the MTSP in Section \ref{sec:bi-obj}.

Next, we present the formulation of the $\varepsilon$-F-MTSP that incorporates fairness in the distribution of tour lengths $\bm l$. Given $\varepsilon \in [0, 1]$, The F-MTSP seeks to enforce the vector of tour lengths to be at least $\varepsilon$-fair using \eqref{eq:soc-e-fair}. This results in the following parameterized mathematical formulation for the $\varepsilon$-F-MTSP:
  \begin{flalign}
    & (\mathcal F^{\varepsilon}): \quad \min \sum_{1 \leqslant v \leqslant m} l_v \quad \text{subject to: \eqref{eq:mtsp-l} -- \eqref{eq:mtsp-l-ge0}, \eqref{eq:soc-e-fair}}. \label{eq:eps-f-mtsp} 
\end{flalign}  
The presence of \eqref{eq:soc-e-fair} in \eqref{eq:eps-f-mtsp} makes the above formulation an MISOCP. In the later sections of this article, we present some theoretical results that will provide intuition on how the $\varepsilon$-F-MTSP enforces the distribution of tour lengths to be fair and develop specialized algorithms to solve the $\varepsilon$-F-MTSP for a fixed value of $\varepsilon$ to global optimality. \\

\subsubsection{$\Delta$-fair multiple traveling salesman problem} \label{subsubsec:delta-f-mtsp}
This variant of the F-MTSP directly enforces an upper bound, denoted by $\Delta \in [0, 1]$, on the Gini coefficient of the tour lengths. We remark that the Gini coefficient in \eqref{eq:gi} has a minimum and maximum value of $0$ and $1$, respectively. Hence, setting an upper bound of $1$ on the Gini coefficient is a trivial upper bound and enforces no fairness in the optimal solution, making it equivalent to the min-sum MTSP. Mathematically, enforcing an upper bound of $\Delta$ on \eqref{eq:gi} is given by 
\begin{gather}
    \sum_{1 \leqslant i < j \leqslant m} | l_i - l_j | \leqslant \Delta \cdot (m-1) \cdot \sum_{i=1}^m l_i. \label{eq:gi-ub}
\end{gather}
Additionally, if the tour lengths of the $m$ salesmen are ordered as 
\begin{gather}
    l_1 \geqslant l_2 \geqslant \cdots \geqslant l_{m-1} \geqslant l_m \label{eq:ordering}
\end{gather}
then the absolute value function in \eqref{eq:gi-ub} can be eliminated to get the following single linear constraint 
\begin{gather}
    \sum_{i = 1}^m (m - 2i + 1) \cdot l_i \leqslant \Delta \cdot (m-1) \cdot \sum_{i=1}^m l_i. \label{eq:linear-gi}
\end{gather}
Note that ordering the tour lengths in \eqref{eq:ordering} is possible for the MTSP as it does not remove any feasible solutions. For other optimization problems, this might not be possible, and in those cases, one can resort to a linear reformulation of the absolute value in \eqref{eq:gi-ub}. Nevertheless, this can lead to two additional constraints for every $(i, j)$ pair in $1 \leqslant i < j \leqslant m$. Using the linear constraint in \eqref{eq:linear-gi}, the parameterized mathematical formulation for the $\Delta$-F-MTSP is given by: 
\begin{flalign}
    & (\mathcal F^{\Delta}): \min \sum_{1 \leqslant v \leqslant m} l_v \text{ subject to: \eqref{eq:mtsp-l} -- \eqref{eq:mtsp-l-ge0}, \eqref{eq:ordering}, \eqref{eq:linear-gi}}. \label{eq:delta-f-mtsp} 
\end{flalign}
Since \eqref{eq:ordering} and \eqref{eq:linear-gi} are linear constraints, the $\Delta$-F-MTSP for any value of $\Delta \in [0, 1]$ continues to be a MILP. Next, we introduce a metric used in the literature to study the trade-off between cost and fairness in optimization problems. 

\subsection{Cost of fairness for the MTSP} \label{subsec:pof} 
Cost of fairness (COF) is a metric used in the literature (see \cite{bertsimas2011price,bektacs2020using}) to study the trade-off between cost and fairness for general combinatorial optimization and resource allocation problems. Here, we use it to study the cost-fairness trade-off for different models of the MTSP. To define COF, let $\mathcal F$ denote any version of the MTSP that focuses on enforcing an equitable distribution of tour lengths \textcolor{black}{and let} $\bm l^*(\mathcal F)$ denote the vector of optimal tour lengths obtained by solving $\mathcal F$. On similar lines, we let $\bm l^*_{\mathrm{min-sum}}$ denote the optimal tour lengths obtained by solving $\mathcal F_1$. Then, the COF is defined as:
\begin{gather}
 \mathrm{COF}(\mathcal F) \triangleq \frac{\|\bm l^*(\mathcal F)\|_1 - \|\bm l^*_{\mathrm{min-sum}}\|_1}{\|\bm l^*_{\mathrm{min-sum}}\|_1}. \label{eq:cof}
\end{gather}
The COF for $\mathcal F$ in \eqref{eq:cof} is the relative increase in the sum of the tour lengths under the fair solution $\bm l^*(\mathcal F)$, compared to the optimal tour lengths obtained by solving \eqref{eq:mtsp}, $\bm l^*_{\mathrm{min-sum}}$. 
Note that COF is $\geqslant 0$
because the $\|\bm l^*_{\mathrm{min-sum}}\|_1$ is the minimum total tour length that any fair version of the MTSP can achieve. The values of $\mathrm{COF}(\mathcal F)$ closer to $0$ are preferable for any $\mathcal F$ because they indicate less compromise on the cost while enforcing fairness in $\bm l^*(\mathcal F)$'s distribution. The following section presents some theoretical properties of the two variants of the F-MTSP.  

\section{Theoretical properties of F-MTSP} \label{sec:properties}
We start by presenting some theoretical properties of the $\varepsilon$-F-MTSP to gain intuition on how $\varepsilon$-F-MTSP enforces fairness in the distribution of tour lengths. We remark that some properties will also extend to the $\Delta$-F-MTSP, and we refrain from proving these results again for the $\Delta$-F-MTSP. 
\subsection{Properties of $\varepsilon$-F-MTSP}
We start by introducing some notations. Let $\mu$ and $\sigma^2$ denote the sample mean and variance of all the tour lengths in $\bm l$. Then, combining the definitions of sample mean, sample variance, and the assumption that all tour lengths are non-negative, we have
\begin{gather}
 \mu = \frac{\|\bm l\|_1}m \text{ and } \sigma^2 = \frac 1{m-1}\left( \|\bm l \|_2^2 - m \mu^2 \right). \label{eq:sample-defn}
\end{gather}
Combining \eqref{eq:soc-e-fair} and \eqref{eq:sample-defn}, we get
\begin{gather}
 CV^2 \triangleq \left(\frac{\sigma}{\mu}\right)^2 \leqslant h(\varepsilon) \label{eq:cv}
\end{gather}
where,
\begin{gather}
 h(\varepsilon) \triangleq \frac{m}{m-1} \left( \frac m{\left(1-\varepsilon + \varepsilon\sqrt{m} \right)^2} - 1\right).
\end{gather}
In \eqref{eq:cv}, $CV$ is the ``coefficient of variation'' in statistics (\cite{Jain1984}), a standardized measure of frequency distribution's dispersion. It is not difficult to see that $h(\varepsilon)$ is strictly decreasing with $\varepsilon$, indicating that as $\varepsilon$ is varied from $0 \rightarrow 1$, the upper bound on $CV$ decreases from a trivial value of $\sqrt m$ to $0$ (enforces all tour lengths to be equal). In other words, $\varepsilon$-fairness enforces fairness by controlling the dispersion in $\bm l$'s distribution through $\varepsilon$.
Next, we present a crucial result from \cite{sundar2023fairly} that enables us to identify trends in the feasibility and optimality of the $\varepsilon$-F-MTSP.
\begin{proposition} \label{prop:eps-max}
Given $\bar\varepsilon \in [0, 1]$, if $\mathcal F^{\bar{\varepsilon}}$ in \eqref{eq:eps-f-mtsp} is infeasible, then it is also infeasible for any $\varepsilon \in [\bar \varepsilon, 1]$.
\end{proposition}
\begin{proof}
 See \cite{sundar2024parametric}.
\end{proof} 
The above states that if the distribution of tour lengths in $\bm l$ for the $\varepsilon$-F-MTSP cannot be made $\bar\varepsilon$-fair, it cannot be made any ``fair-er'', aligning with our intuitive understanding of fairness. A practical consequence of the proposition is that there exists a $0\leqslant \varepsilon^{\max}\leqslant 1$ such that for any $\varepsilon \in (\varepsilon^{\max}, 1]$, $\mathcal F^{{\varepsilon}}$ is infeasible, and for any $\varepsilon \in [0,\varepsilon^{\max}]$, $\mathcal F^{{\varepsilon}}$ is feasible.
Throughout the rest of the article, we shall refer to $[0, \varepsilon^{\max}]$ as the ``feasibility domain'' for $\varepsilon$-F-MTSP. 

The next result is aimed at the qualitative behavior of the COF in \eqref{eq:cof} for the $\varepsilon$-F-MTSP. We first recall that $\bm l^*(\mathcal F)$ denotes the vector of optimal tour lengths obtained by solving some version of the MTSP that focuses on enforcing fairness in the distribution of tour lengths. Then, we have the following result:
\begin{proposition} \label{prop:monotonicity}
Given an instance of the $\varepsilon$-F-MTSP, the univariate function $\|\bm l^*(\mathcal F^\varepsilon)\|_1$ monotonically increases with $\varepsilon$ for all values of $\varepsilon$ in the feasibility domain of the instance. 
\end{proposition}
\begin{proof}
 The proof follows from the following three observations:
 \begin{enumerate}
  \item the feasible set of solutions of $\mathcal F^{\varepsilon}$ becomes smaller as $\varepsilon$ increases within its feasibility domain (see \cite{sundar2023fairly}), 
  \item $\mathcal F^\varepsilon$ is a minimization problem and 
  \item the objective function of the $\mathcal F^{\varepsilon}$ can be rewritten in terms of the 1-norm as $\|\bm l\|_1$, where $\bm l$ is the vector of decision variables $l_v$, $1 \leqslant v \leqslant m$.
 \end{enumerate}
 Together, the above observations lead to the conclusion that if $\varepsilon_1 \geqslant \varepsilon_2$, then $\|\bm l^*(\mathcal F^{\varepsilon_1})\|_1 \geqslant \|\bm l^*(\mathcal F^{\varepsilon_2})\|_1$. 
\end{proof}
Proposition \ref{prop:monotonicity} tells us that when $\varepsilon$ increases, the distribution of $\bm l$ becomes fairer, but the cost (sum of the tour lengths), given by $\|\bm l\|_1$, increases. This conforms with our intuitive understanding of the trade-off between fairness and cost and enables quantifying this trade-off for different values of $\varepsilon$. A useful corollary of Proposition \ref{prop:monotonicity} is 
\begin{corollary} \label{cor:cof}
    $\mathrm{COF}(\mathcal F^{\varepsilon})$ in \eqref{eq:cof} monotonically increases with $\varepsilon$ for all values $\varepsilon$ in the feasibility domain of the $\varepsilon$-F-MTSP instance.
\end{corollary}

The final theoretical result establishes a closed-form relationship between the Jain et al. index in \cite{Jain1984} and the value of $\varepsilon$ by showing that enforcing $\varepsilon$-fairness for the tour lengths is equivalent to setting the Jain et al. index of the tour lengths to a bijective function of $\varepsilon$. This will translate any value of the Jain et al. index sought to an equivalent $\varepsilon$ value to enforce $\varepsilon$-fairness. To the best of our knowledge, this cannot be achieved using any existing fairness model.
\begin{proposition} \label{prop:ji}
Enforcing the vector of tour lengths $\bm l  \in \mathbb R_{\geqslant 0}^m$ to be $\varepsilon$-fair is equivalent to setting $$\mathrm{JI}(\bm l) = \frac{(1-\varepsilon+\varepsilon\sqrt m)^2}m$$
\end{proposition}
\begin{proof}
 Utilizing the non-negativity of the tour lengths, Jain et al. index in \eqref{eq:ji} can be equivalently rewritten as $$\mathrm{JI}(\bm l) = \frac{1}{m} \cdot \frac{\left( \sum_{i=1}^{m} l_i \right)^2}{\sum_{i=1}^{m} l_i^2} = \frac{\|\bm l\|_1^2}{m\cdot \|\bm l\|_2^2}.$$ Combining the above equation with the definition of enforcing $\varepsilon$-fairness in $\bm l$ leads to
 \begin{gather}
  \mathrm{JI}(\bm l) =\frac{\left(1-\varepsilon + \varepsilon\sqrt{m} \right)^2}{m} \triangleq w(\varepsilon). \label{eq:w-defn}
 \end{gather}
It is also easy to observe that enforcing $\bm l$ to be at least $\varepsilon$-fair is equivalent to $\mathrm{JI}(\bm l) \geqslant w(\varepsilon)$. 
\end{proof}
Proposition \ref{prop:ji} establishes a bijective function, $w(\varepsilon)$, that maps any value of $\varepsilon$ to an equivalent value of the Jain et al. index. Notably, the similarity between the terms used in the definitions of the Jain et al. index in \eqref{eq:ji} and the $\varepsilon$-Fair index in \eqref{eq:efi} is not coincidental. The Jain et al. index in \eqref{eq:ji} was originally hypothesized to take this form and was shown to possess desirable properties such as independence from population size, scale, and metric invariance, boundedness, and continuity with respect to small changes in tour lengths (\cite{Jain1984}). However, this development in \cite{Jain1984} did not provide a clear mathematical derivation or justification for why it effectively measures fairness. In contrast, the development of $\varepsilon$-fairness in Section \ref{sec:mtsp}, rooted in a fundamental result from linear algebra, offers a rigorous foundation for understanding the Jain et al. index.

Next, we present similar theoretical properties for the $\Delta$-F-MTSP. We remark that the proofs for the $\Delta$-F-MTSP are omitted since the arguments are very similar to the respective proofs for the $\varepsilon$-F-MTSP. 

\subsection{Properties of $\Delta$-F-MTSP}
Similar to the $\varepsilon$-F-MTSP that enforces fairness by controlling the upper bound on the coefficient of variation $CV$, the $\Delta$-F-MTSP enforces fairness on the distribution of tour lengths by controlling the sum of pairwise absolute differences. Since the Gini coefficient is based on absolute differences rather than squared values like the coefficient of variation, one can expect the fairness enforced by $\varepsilon$-F-MTSP to be more influenced by outliers (excessively low or high values)--\cite{de2007income}. We now present results analogous to Propositions \ref{prop:eps-max} and \ref{prop:monotonicity} for the $\Delta$-F-MTSP. 
\begin{proposition} \label{prop:delta-min}
    Given $\bar \Delta \in [0, 1]$, if $\mathcal F^{\bar \Delta}$ in \eqref{eq:delta-f-mtsp} is infeasible, then it is also infeasible for any $\Delta \in [0, \bar\Delta]$.
\end{proposition}
The above proposition states that if the $\Delta$-F-MTSP is infeasible when the upper bound on the Gini coefficient is set to $\bar\Delta$, then this bound cannot be reduced further. Similar to $\varepsilon$-F-MTSP, this proposition can be used to define a $\Delta^{\min}$ such that for any $\Delta \in [0, \Delta^{\min})$, $\Delta$-F-MTSP is infeasible and feasible for any $\Delta \in [\Delta^{\min}, 1]$ resulting in $[\Delta^{\min}, 1]$ being referred to as the ``feasibility domain'' for $\Delta$-F-MTSP.
\begin{proposition} \label{prop:delta-monotonicity}
    Given an instance of the $\Delta$-F-MTSP, the univariate function $\|\bm l^*(\mathcal F^{\Delta})\|_1$ and $\mathrm{COF}(\mathcal F^{\Delta})$ in \eqref{eq:cof} monotonically decreases with $\Delta$ for all values of $\Delta$ in the feasibility domain of the instance. 
\end{proposition}
\begin{proof}
    Similar to Proposition \ref{prop:monotonicity}.
\end{proof}
In the next section, we present algorithmic approaches to solve all the variants of the MTSP that seek to enforce fairness to global optimality. 

\section{Algorithms} \label{sec:algo}
The state-of-the-art algorithm to solve both the min-sum MTSP and min-max MTSP to optimality is the branch-and-cut algorithm (see \cite{sundar2017algorithms}). Though the branch-and-cut algorithm is well-studied in the literature for a more general class of MVRPs (see \cite{toth2002vehicle}), we present the main ingredients of the algorithm for keeping the presentation self-contained; for a complete pseudo-code of the approach, interested readers are referred to \cite{sundar2015exact,sundar2017algorithms}. The same algorithm, without modification, is used to solve the $\Delta$-F-MTSP to optimality, since its formulation is the same as that of the min-sum MTSP, except for \eqref{eq:ordering} and \eqref{eq:linear-gi}. Furthermore, all the basic ingredients of the branch-and-cut algorithm will also be the building blocks for a custom and enhanced branch-and-cut algorithm, shown in the subsequent paragraphs, to solve the $p$-norm MTSP and the $\varepsilon$-F-MTSP. 

\subsection{Separation of sub-tour elimination constraints in \eqref{eq:mtsp-sec}} \label{subsec:sec}
We first remark that if not for the sub-tour elimination constraints in \eqref{eq:mtsp-sec} that are exponential in the number of targets, the min-sum or min-max MTSP can directly be provided to an off-the-shelf commercial or open-source MILP solver that uses a standard implementation of the branch-and-cut algorithm to solve the problems to optimality. \textcolor{black}{Since the number of sub-tour elimination constraints scales exponentially with the number of targets, explicitly encoding them for the solver becomes intractable}. The standard solution to this difficulty is first to solve the relaxed problem, i.e., the problem without \eqref{eq:mtsp-sec}, and add sub-tour elimination constraints only when the solution to the relaxed problem violates them. The algorithms that identify the subset of targets for which the sub-tour elimination constraints are violated, given any solution to the relaxed problem, are called separation algorithms (see \cite{toth2002vehicle}). We now present the separation algorithms that can identify a subset of targets ($S$ in \eqref{eq:mtsp-sec}) for which \eqref{eq:mtsp-sec} is violated, given either an integer or a fractional feasible solution to the relaxed problem.  

Before presenting the separation algorithms, we introduce notation. For every salesman $v$, we let $G^*_v = (V^*_v, E^*_v)$ denote the support graph associated with a given fractional or integer solution to the relaxed problem $(\bm x^*, \bm y^*)$, i.e., $V^*_v \triangleq \{ i\in V: y_i^v > 0 \}$ and $E^*_v \triangleq \{e \in E: x_e^v > 0\}$. Here, $\bm x^*$ and $\bm y^*$ are the vectors of the decision variable values in \eqref{eq:mtsp-edge-depot}--\eqref{eq:mtsp-vertex}. For every salesman $v$, the violation of the inequality \eqref{eq:mtsp-sec} can be verified by examining the connected components in $G^*_v$. Each connected component $C$ that does not contain the depot $d$ generates a
violated sub-tour elimination constraint for $S = C$ and each $i \in S$. If a connected component $C$ contains the depot $d$, the following procedure is used to find the one or more violated sub-tour elimination constraints \eqref{eq:mtsp-sec} for salesman $v$. 
Given a connected component $C$ for salesman $v$ that contains a depot $d$ and a fractional solution $(\bm x^*, \bm y^*)$, the most violated constraint of the form \eqref{eq:mtsp-sec} can be obtained by computing a global min-cut (\cite{gomory1961multi}) on a capacitated, undirected graph $\bar G_v = (C, E^*_v)$, with each edge in $e \in E^*_v$ assigned the capacity of $\bm x_e^{v*}$. The global min-cut yields a disjoint partition of the set $C = C_1 \bigcup C_2$ with $C_1$ containing the depot. If the value of this min-cut is strictly less than $2 \cdot y_i^{v*}$ for $i \in C_2$, then the sub-tour elimination constraint corresponding to the salesman $v$, the set $S = C_2$, and $i \in C_2$ is violated. This constraint or set of constraints is then added to the relaxed problem, and the problem is re-solved. 

The branch-and-cut approach equipped with the above separation algorithm for dynamically generating sub-tour elimination constraints can be directly used to solve the min-sum MTSP, min-max MTSP, and $\Delta$-F-MTSP for any $\Delta \in [0, 1]$, as all of these variants are formulated as MILPs. Application of this approach to the $p$-norm MTSP (with $1 < p < \infty$) and the $\varepsilon$-F-MTSP (with $\varepsilon > 0$) requires enhancements to handle the convex $\ell^p$ norm of the tour lengths in the objective function of \eqref{eq:p-norm} and the SOC in \eqref{eq:soc-e-fair}, respectively. When $p \in \{1, \infty\}$, the $p$-norm MTSP reduces to the min-sum and min-max MTSP, respectively, and when $\varepsilon = 0$, the $\varepsilon$-F-MTSP is equivalent to the min-sum MTSP. In the subsequent sections, we present an outer-approximation technique that dynamically adds linear outer approximations of the $\ell^p$ norm and the SOC constraint when a solution to the relaxed problem, which ignores these constraints, violates them. We start with the $p$-norm MTSP in the next section. 

\subsection{$p$-norm MTSP} \label{subsec:p-norm-algo}
To outer approximate the $\ell^p$ norm of the tour lengths in the objective function, we first rewrite \eqref{eq:p-norm} equivalently as follows: 
\begin{subequations}
\begin{gather} 
    (\mathcal F_p): ~ \min z \text{ subject to: \eqref{eq:mtsp-l} -- \eqref{eq:mtsp-l-ge0}, and }  \\ 
    z^p \geqslant \sum_{1 \leqslant v \leqslant m} l_v^p. \label{eq:p-norm-cons} 
\end{gather}
    \label{eq:p-norm-eq}
\end{subequations}
Notice that the additional constraint \eqref{eq:p-norm-cons} is a convex constraint due to the convexity of any vector's $\ell^p$ norm. In particular, the constraint defines a convex, $p$-norm cone of dimension $(m+1)$ i.e., $P = \{(z, \bm l) \in \mathbb R^{m+1}: z \geqslant \| \bm l \|_p \}$ (see \cite{boyd2004convex}). It is known in existing literature that when convex cones are outer approximated using linear inequalities, it is always theoretically stronger and also computationally efficient to equivalently represent a high-dimensional cone into potentially many smaller-dimensional cones and outer approximate each of the small cones instead of outer approximating the high-dimensional cone directly (see \cite{lubin2018polyhedral}). Converting a high-dimensional conic constraint into many low-dimensional cones is called ``cone disaggregation'' (see \cite{lubin2018polyhedral}). To that end, we now present the sequence of steps to disaggregate the $(m+1)$ dimensional $p$-norm cone in \eqref{eq:p-norm-cons} to an $m$ three-dimensional convex cone and one single linear constraint as follows:
\begin{subequations}
\begin{gather*}
     z^p \geqslant \sum_{1 \leqslant v \leqslant m} l_v^p ~~\equiv~~ z \geqslant \sum_{1 \leqslant v \leqslant m} \left(\frac{l_v^p}{z^{p-1}}\right) \label{eq:p-norm-cone-1} \\ 
     \equiv \left\{ z = \sum_{1 \leqslant v \leqslant m} L_v \text{ and } L_v \cdot z^{p-1} \geqslant l_v^p ~~\forall 1 \leqslant v \leqslant m \right\} \\
     \equiv \left\{ z = \sum_{1 \leqslant v \leqslant m} L_v \text{ and } L_v^{\frac 1p} \cdot z^{1-\frac 1p} \geqslant l_v ~~\forall 1 \leqslant v \leqslant m \right\}
\end{gather*}
\end{subequations}
\textcolor{black}{where $L_v$, for each $v$, is an auxiliary variable representing ${l_v^p}/{z^{p-1}}$, yielding the equivalent conic constraints $L_v \cdot z^{p-1} \geqslant l_v^p$.}
In the last equivalence, the constraint $L_v^{\frac 1p} \cdot z^{1-\frac 1p} \geqslant l_v$ is convex for every $v$ and $1 < p < \infty$; in fact the constraint defines a three dimensional power cone (\cite{boyd2004convex}) defined as
\begin{flalign}
    \mathcal P_3^{\alpha, 1-\alpha} \triangleq \{ (x_1, x_2, x_3) \in \mathbb R_{\geqslant 0}^3: x_1^{\alpha} \cdot x_2^{1-\alpha} \geqslant x_3 \}. \label{eq:power-cone}
\end{flalign}
Using the equivalence, \eqref{eq:p-norm-cons} which is a $(m+1)$-dimensional $p$-norm cone can be rewritten as 
\begin{flalign}
    \left\{ z = \sum_{1 \leqslant v \leqslant m} L_v \text{ and } (L_v, z, l_v) \in \mathcal P_3^{1/p, 1-1/p} ~~\forall v \right\} \label{eq:dis-p}
\end{flalign}
that has $m$ three-dimensional power cones and one linear equality constraint. The linear outer approximation of the three-dimensional power cone is obtained as the first-order Taylor expansion of the power cone at any point $(L_v^0, z^0, l_v^0)$ that lies on the surface of the cone, i.e., $(L_v^0, z^0, l_v^0)$ satisfies the conic constraint at equality. The linear outer approximation of $(L_v, z, l_v) \in \mathcal P_3^{1/p, 1-1/p}$ at $(L_v^0, z^0, l_v^0)$ given by 
\begin{flalign}
    \left(\frac 1p\right) l_v^0 \cdot z^0 \cdot L_v + \left(1-\frac 1p\right) l_v^0 \cdot L_v^0 \cdot z \geqslant L_v^0 \cdot z^0 \cdot l_v. \label{eq:oa-p}
\end{flalign}
In the branch-and-cut algorithm, we first relax the $m$ three-dimensional power cones and add linear outer approximations of these constraints at the point on the cone's surface when the optimal solution to the relaxed problem violates any of them. The point on the cone's surface is obtained by projecting the violated solution onto the cone's surface. This process is pictorially shown in Figure \ref{fig:oa}. 
\begin{figure}
    \centering
    \includegraphics[scale=0.2,frame]{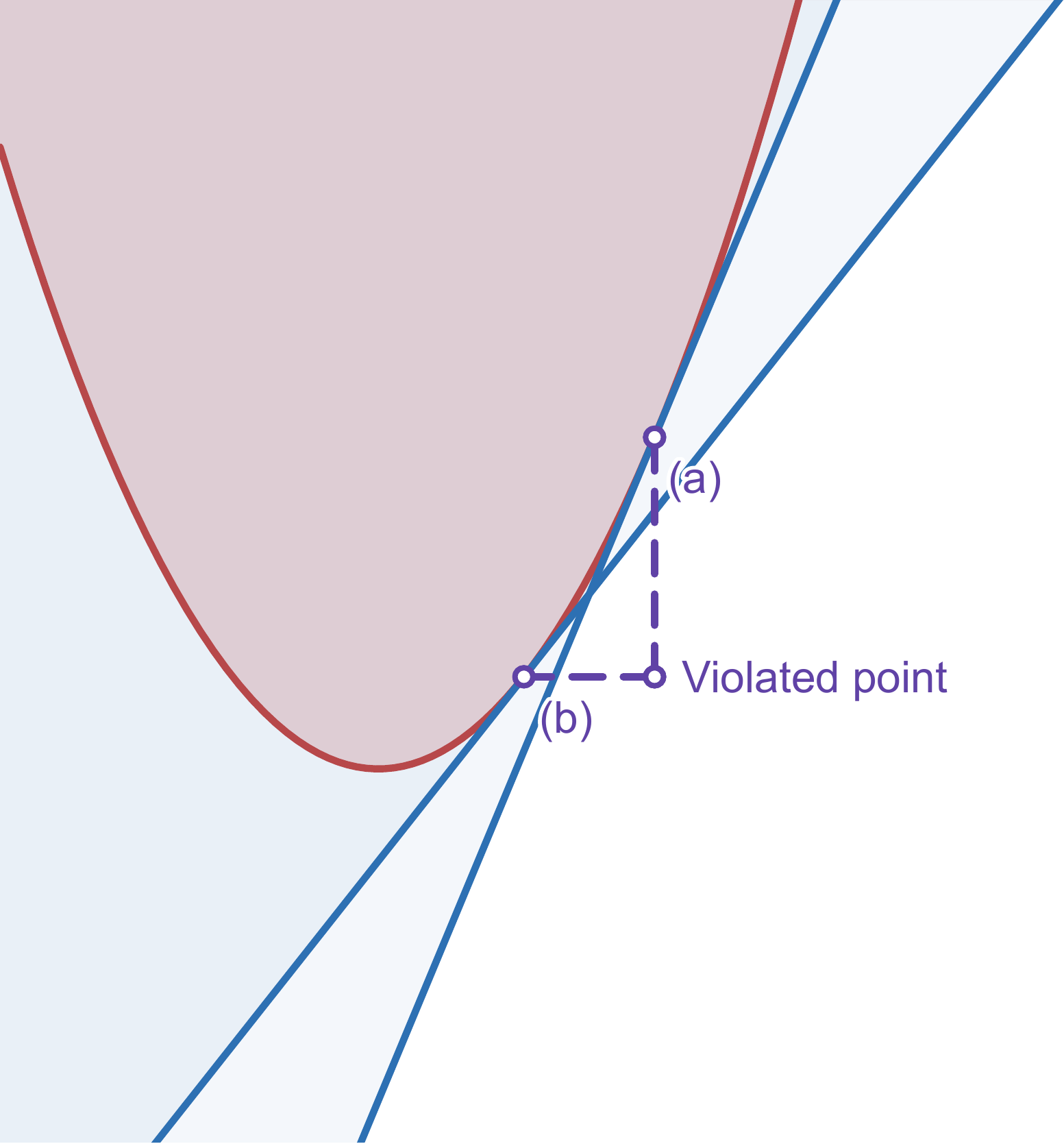}
    \caption{An illustration of the outer approximation procedure. The region shaded in red is the convex constraint, and the violated point is the solution obtained when solving the relaxed problem. This point can be projected onto the surface of the convex curve in many ways; two trivial projections (a) and (b) are shown. For each point, a tangent outer approximates the convex region (shaded in red). }
    \label{fig:oa}
\end{figure}

\subsection{$\varepsilon$-F-MTSP} \label{subsec:f-mtsp-algo}
For the $\varepsilon$-F-MTSP, the SOC constraint that enforces $\varepsilon$-fairness in the distribution of tour lengths in \eqref{eq:soc-e-fair} is outer approximated analogous to the procedure for outer approximating the $p$-norm cone constraint in \eqref{eq:p-norm-cons}. The procedure is the same if we introduce an additional variable $z$ and equivalently represent $\mathcal F^{\varepsilon}$ as 
\begin{subequations}
\begin{gather}
    (\mathcal F^{\varepsilon}): \quad \min \sum_{1 \leqslant v \leqslant m} l_v \quad \text{subject to: \eqref{eq:mtsp-l} -- \eqref{eq:mtsp-l-ge0}, } \\ 
    z \cdot \left(1-\varepsilon + \varepsilon\sqrt{m} \right) = \sum_{1 \leqslant v \leqslant m} l_v, \text{ and }
    z \geqslant \|\bm l\|_2. \label{eq:eps-mtsp-2-norm}
\end{gather}
\label{eq:eps-mtsp-eq}
\end{subequations} 
We remark that the SOC constraint in \eqref{eq:eps-mtsp-2-norm} is exactly equivalent to the $2$-norm constraint that is obtained when letting $p = 2$ in \eqref{eq:p-norm-cons} and hence, the same procedure detailed in the previous section is used to outer approximate the SOC constraint. In the next section, we formulate two bi-objective versions of the MTSP that incorporate length balancing and show how they relate to the $\varepsilon$-fair MTSP and $\Delta$-fair MTSP. 

\section{Bi-objective MTSP} \label{sec:bi-obj}
Thus far, we have formulated different variants of the MTSP that incorporate fairness in the distribution of tour lengths while minimizing the overall cost. In this section, we take a different approach by explicitly formulating two bi-objective MTSP variants, in which one objective minimizes the total tour length (``min-sum''), and the other enforces fairness in the tour-length distribution. Similar to the two fairness-based MTSP variants, $\varepsilon$-F-MTSP and $\Delta$-F-MTSP, we introduce two bi-objective formulations:

\begin{flalign}
    & (\mathcal F^{\varepsilon}_{\mathrm{bi-obj}}): \quad \min \left( \sum_{1 \leqslant v \leqslant m} l_v, ~ -\mathrm{\varepsilon FI}(\bm{l}) \right) \quad \text{subject to: \eqref{eq:mtsp-l} -- \eqref{eq:mtsp-l-ge0}}. \label{eq:eps-f-mtsp-biobj} 
\end{flalign}  

\begin{flalign}
    & (\mathcal F^{\Delta}_{\mathrm{bi-obj}}): \quad \min \left( \sum_{1 \leqslant v \leqslant m} l_v, ~ \mathrm{GC}(\bm{l}) \right) \quad \text{subject to: \eqref{eq:mtsp-l} -- \eqref{eq:mtsp-l-ge0}}. \label{eq:delta-f-mtsp-biobj} 
\end{flalign}  

In \eqref{eq:eps-f-mtsp-biobj}, the second objective aims to maximize $\mathrm{\varepsilon FI}(\bm{l})$, which is equivalent to minimizing $-\mathrm{\varepsilon FI}(\bm{l})$. This enforces fairness in the tour-length distribution, as a higher $\varepsilon$-fair index corresponds to a more balanced workload allocation among the salesmen. Similarly, in \eqref{eq:delta-f-mtsp-biobj}, fairness is promoted by minimizing the Gini coefficient of $\bm{l}$, since a lower Gini coefficient indicates a more equitable distribution of tour lengths.

A well-established approach to obtaining the Pareto front of bi-objective optimization problems is scalarization, which converts the multi-objective problem into a single-objective problem (\cite{hwang2012multiple}). There are several scalarization techniques (see \cite{miettinen1999nonlinear}); in this work, we use the $\alpha$-constraint method--\cite{miettinen1999nonlinear}. This approach enforces a lower or upper bound (denoted by $\alpha$) on one of the objectives, depending on whether the objective is to be maximized or minimized.
The following proposition establishes that applying the $\alpha$-constraint method to the fairness objectives in \eqref{eq:eps-f-mtsp-biobj} and \eqref{eq:delta-f-mtsp-biobj} results in the fairness-based MTSP formulations $\mathcal{F}^{\varepsilon}$ and $\mathcal{F}^{\Delta}$ for appropriate values of $\varepsilon$ and $\Delta$, respectively.

\begin{proposition} \label{prop:scalarization}
    Scalarizing $\mathcal F^{\varepsilon}_{\mathrm{bi-obj}}$ and $\mathcal F^{\Delta}_{\mathrm{bi-obj}}$ using the $\alpha$-constraint method (see \cite{miettinen1999nonlinear}) results in the optimization problems $\mathcal F^{\varepsilon}$ and $\mathcal F^{\Delta}$ for appropriate values of $\varepsilon$ and $\Delta$, respectively. 
\end{proposition}

\begin{proof}
    We present the proof for $\mathcal F^{\varepsilon}_{\mathrm{bi-obj}}$; the proof for $\mathcal F^{\Delta}_{\mathrm{bi-obj}}$ follows similarly and is omitted for brevity. Applying the $\alpha$-constraint method to the fairness objective in $\mathcal F^{\varepsilon}_{\mathrm{bi-obj}}$ results in the following optimization problem:
    \begin{flalign*}
    & ~~ \min \sum_{1 \leqslant v \leqslant m} l_v \quad \text{subject to: \eqref{eq:mtsp-l} -- \eqref{eq:mtsp-l-ge0}, \text{ and } } -\mathrm{\varepsilon FI}(\bm{l}) \leqslant \alpha \\ 
    \Rightarrow & ~~ \min \sum_{1 \leqslant v \leqslant m} l_v \quad \text{subject to: \eqref{eq:mtsp-l} -- \eqref{eq:mtsp-l-ge0}, \text{ and } } -\alpha \leqslant \frac{1}{\sqrt{m}-1} \left(\frac{\|\bm{l}\|_1}{\|\bm{l}\|_2} -1 \right) \\ 
    \Rightarrow & ~~ \min \sum_{1 \leqslant v \leqslant m} l_v \quad \text{subject to: \eqref{eq:mtsp-l} -- \eqref{eq:mtsp-l-ge0}, \text{ and } } \left( 1 - (-\alpha) + (-\alpha) \sqrt{m} \right) \cdot \|\bm l\|_2 \leqslant \|\bm l\|_1 \\ 
    \equiv & ~~ (\mathcal F^{\varepsilon}) \text{ with $\varepsilon = -\alpha$}.  
    \end{flalign*}  
    The first and second implications follow from Definition \ref{def:eps-fairness}. During the scalarization process, $\alpha$ must be chosen to be consistent with the feasibility of the constraint $-\mathrm{\varepsilon FI}(\bm{l}) \leqslant \alpha$, i.e.,  $\alpha \in [-1, 0]$. When $\alpha = 0$, no fairness on the distribution of tour lengths is enforced, while $\alpha = -1$ enforces complete fairness by ensuring equal tour lengths, aligning with the development of $\varepsilon$-fairness in Section \ref{sec:mtsp}.
\end{proof}
By solving the scalarized problem for all feasible values of $\alpha$, we can compute the Pareto front, which characterizes the trade-off between fairness in tour length distribution and total tour length (a measure of efficiency). Furthermore, the solution algorithms presented in Section \ref{sec:algo} can be directly applied to solve these scalarized formulations and obtain the Pareto front. To the best of our knowledge, this is the first work in the literature that explicitly enables the computation of the Pareto front between fairness and efficiency in the MTSP.
In the next section, we present the results of extensive computational experiments that elucidate the pros and cons of every variant of the MTSP presented in this article. 

\section{Computational results} \label{sec:results}
In this section, we present the computational results for the proposed algorithm to solve Fair-MTSP and p-norm MTSP, and we show its advantages over min-max MTSP. We begin by providing details of the test instances. A total of $59$ instances were chosen for the computational experiments. 36 instances were derived from 12 TSPLIB benchmark instances in \cite{reinelt1991tsplib} that consists of various single vehicle TSP instances, 22 instances were taken from the CVRPLIB benchmark library in \url{https://galgos.inf.puc-rio.br/cvrplib/index.php/en/instances} and one instance with 50 targets was generated using the road network data in the suburb of Seattle using OpenStreetMap to showcase the practicality of this work; we refer to this instance as ``Seattle''. 
For all 12 TSPLIB instances, the depot location was chosen as the centroid of the target locations, and the number of salesmen $m$ was set to $ m \in \{3, 4, 5\}$. This yielded a total of 36 instances, derived from 12 TSPLIB instances. The CVRPLIB benchmark instances are multi-vehicle, capacitated vehicle routing problems with depot locations. For the 22 instances chosen from this library, the vehicle capacities were ignored, and the number of salesmen $m$ was set to the number of vehicles. The final 50-target instance, ``Seattle'', was designed to show the usefulness of the formulations proposed in this paper in a practical setting. We refer the reader to Section \ref{subsec:real} for more details on this particular instance. All algorithms were implemented in Kotlin, with CPLEX 22.1 as the underlying branch-and-cut solver, and all computational experiments were run on an Intel Haswell 2.6 GHz, 32 GB, 20-core machine. Furthermore, we impose a computational time limit of 1 hour on any run of the branch-and-cut algorithms presented in this paper. 
Finally, we remark that the implementations of the algorithms for all variants of the MTSP and the F-MTSP are open-sourced and available at \texttt{\url{https://github.com/kaarthiksundar/FairMTSP}}.

\subsection{Computation time study} \label{subsec:time} 
This study reports the results for the 36 TSPLIB and 22 CVRPLIB instances. The instance-by-instance results can be found in the \ref{app:runtime}. In particular, the Tables \ref{tab:pNorm}, \ref{tab:epsFair}, and \ref{tab:deltaFair} summarize the computation times for the different variants of the MTSP formulated in this article for the TSPLIB instances, and the Tables \ref{tab:pNorm_vrp}, \ref{tab:epsFair_vrp}, and \ref{tab:deltaFair_vrp} summarize the computation times for the CVRPLIB instances. The naming convention used in the Tables for the instances is ``name-m'' where ``name``corresponds to the name of the TSPLIB or CVRPLIB instance and $m$ represents the number of salesmen. For each of the 58 instances, we solve the min-sum MTSP ($\mathcal F_1$), the min-max MTSP ($\mathcal F_{\infty}$), the $p$-norm MTSP ($\mathcal F_p$) for $p \in \{2, 3, 5, 10\}$, the $\varepsilon$-F-MTSP ($\mathcal F^{\varepsilon}$) for $\varepsilon \in \{0.1, 0.3, 0.5, 0.7, 0.9\}$ and the $\Delta$-F-MTSP ($\mathcal F^{\Delta}$) for $\Delta \in \{0.1, 0.3, 0.5, 0.7, 0.9\}$. The tables report the computation times for all aforementioned variants of the MTSP in seconds. If the computation time exceeds the one-hour limit, the corresponding entry is marked ``TO'' to indicate that the run timed out. For each instance labeled TO, the number in parentheses represents the optimality gap at termination. 

\textcolor{black}{The statistics of the results in \ref{app:runtime} are provided in Table \ref{tab:run-time-stats}. The column ``solved'' reports how many instances (out of 58) were completed within the computational time limit, and ``timeout'' gives the fraction that exceeded it. Runtime statistics are computed over solved instances, where ``median'' denotes the typical runtime, while ``P25'' and ``P75'' denote the lower and upper quartiles, respectively, capturing easier and harder cases. Finally, ``PAR'' (see \eqref{eq:par1}) denotes the penalized average runtime, in which unsolved instances are assigned a penalty equal to the time limit, yielding a single metric that accounts for both speed and robustness.
\begin{equation}
    \mathrm{PAR} \;=\; \frac{1}{N} \sum_{i=1}^{N}
    \begin{cases}
    t_i, & \text{if instance } i \text{ is solved}, \\
    3600, & \text{if instance } i \text{ times out}.
    \end{cases} \label{eq:par1}
\end{equation}}

\setlength{\tabcolsep}{10pt}
\begin{table}[htbp]
    \centering
    \caption{Runtime statistics over 58 instances.}
    \label{tab:run-time-stats}
    \addtolength{\tabcolsep}{-0.4em}
    \begin{tabular}{lrrrrrr}
     \toprule
    \multirow{2}{*}{formulation} & \multirow{2}{*}{\# solved} & \multirow{2}{*}{timeout (\%)} & \multicolumn{4}{c}{computation time in seconds} \\
    \cmidrule{4-7}
     &  & & median & P25 & P75 & PAR \\
    \midrule 
    \begin{filecontents*}{tables/runtime_stats.csv}
formulation,solved,timeout,median,p25,p75,par1
$\mathcal F_1$,58,0.0,3.64,1.01,9.59,8.58
$\mathcal F_{\infty}$,32,44.83,27.68,8.42,285.03,1810.94
$\mathcal F_p$ ($p=2$),47,18.97,45.74,7.48,350.71,994.81
$\mathcal F_p$ ($p=3$),43,25.86,54.31,7.33,252.52,1159.38
$\mathcal F_p$ ($p=5$),42,27.59,50.42,6.89,356.51,1244.57
$\mathcal F_p$ ($p=10$),40,31.03,35.44,8.2,369.58,1368.69
$\mathcal F_{\varepsilon}$ ($\varepsilon=0.1$),58,0.0,13.28,4.3,39.47,35.79
$\mathcal F_{\varepsilon}$ ($\varepsilon=0.3$),58,0.0,16.18,4.5,43.43,65.89
$\mathcal F_{\varepsilon}$ ($\varepsilon=0.5$),54,6.9,23.38,6.82,71.47,351.27
$\mathcal F_{\varepsilon}$ ($\varepsilon=0.7$),52,10.34,57.24,10.09,175.37,635.58
$\mathcal F_{\varepsilon}$ ($\varepsilon=0.9$),46,20.69,64.16,8.49,197.41,970.59
$\mathcal F_{\Delta}$ ($\Delta=0.1$),42,27.59,44.56,8.32,224.66,1225.03
$\mathcal F_{\Delta}$ ($\Delta=0.3$),47,18.97,46.65,7.57,205.09,850.04
$\mathcal F_{\Delta}$ ($\Delta=0.5$),50,13.79,60.13,7.52,180.12,708.1
$\mathcal F_{\Delta}$ ($\Delta=0.7$),55,5.17,24.11,5.15,83.45,279.01
$\mathcal F_{\Delta}$ ($\Delta=0.9$),56,3.45,11.76,5.06,40.97,159.13
\end{filecontents*}
\csvreader[head to column names, late after line=\\]{tables/runtime_stats.csv}{}
    {\csvcoli & \csvcolii & \csvcoliii & \csvcoliv & \csvcolv & \csvcolvi & \csvcolvii}
    \bottomrule
\end{tabular}
\end{table}

\textcolor{black}{The summary of the results is as follows: 
\begin{enumerate}[label=(\roman*)]
    \item The min-max MTSP is the most computationally demanding formulation, exhibiting the longest runtimes, the largest spread, and the worst robustness (higher timeout rates and PAR), which supports the motivation for alternative fairness-driven formulations. 
    \item The $p$-norm MTSP performs comparably to the other fairness-based formulations, $\varepsilon$-F-MTSP and $\Delta$-F-MTSP, on easier instances (similar P25 and median runtimes), but shows degraded performance on harder instances, as reflected in higher P75 values and PAR, indicating weaker scalability. In $p$-norm MTSP formulations, there is mild sensitivity to the choice of $p$. As $p$ increases, the formulations become progressively more difficult to solve, as evidenced by fewer solved instances, higher timeout rates, and worsening PAR values. However, central runtime statistics (median, P25, P75) do not exhibit a clear monotonic trend, indicating that the impact of $p$ is more pronounced in robustness than in typical-case runtime.
    \item  The $\varepsilon$-F-MTSP and $\Delta$-F-MTSP exhibit similar central tendencies and variability (median, P25, P75), confirming no clear dominance in raw runtime. However, the $\varepsilon$-F-MTSP demonstrates better robustness, with a lower timeout percentage and improved PAR, making it preferable when solution reliability within the time limit is the primary concern. This is also evident from the detailed results reported in the  Table pairs (\ref{tab:epsFair},  \ref{tab:deltaFair}) and (\ref{tab:epsFair_vrp}, \ref{tab:deltaFair_vrp}).
\end{enumerate}}

\subsection{Pareto front and cost of fairness examination} \label{subsec:cof}
For this study, we focus on the ``eil51'' instance with $m = 5$. The omission of other instances is intentional, as the trends observed for ``eil51'' are consistent across all instances. The choice of this particular instance is also arbitrary. The objective of this section is to compute the Pareto front for the bi-objective MTSP formulations $\mathcal{F}_{\mathrm{bi-obj}}^{\varepsilon}$ and $\mathcal{F}_{\mathrm{bi-obj}}^{\Delta}$. To achieve this, we apply the scalarization technique outlined in Proposition \ref{prop:scalarization}, varying the values of $\alpha$ to construct the Pareto front. The Pareto fronts for $ \mathcal{F}^{\varepsilon}_{\mathrm{bi-obj}} $ and $ \mathcal{F}^{\Delta}_{\mathrm{bi-obj}} $ are presented in Figure \ref{fig:pareto}. From the plot, it is evident that the "min-sum" and "min-max" solutions represent the two extreme cases for the MTSP. 

\begin{figure}[htbp]
    \centering
    \includegraphics[scale=1.0]{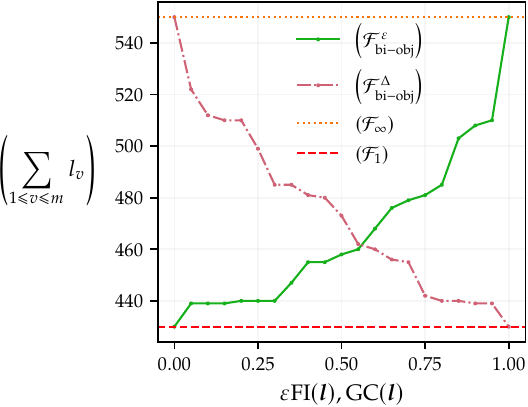}
    \caption{The Pareto fronts for the two bi-objective formulations, the total tour length for the ``min-sum'' and ``min-max'' MTSP for the eil51 instance with $m=5$ are shown.}
    \label{fig:pareto}
\end{figure}

Next, we solve the min-sum MTSP, min-max MTSP, $\varepsilon$-F-MTSP and the $\Delta$-F-MTSP for $\varepsilon$ and $\Delta$ values starting from $0$ to $1$ in steps of $0.05$; for both the min-max and the fair versions of the MTSP, we compute the cost of fairness as defined by \eqref{eq:cof}. Figure \ref{fig:cof} shows the COF values as we increase the values of $\varepsilon$ and $\Delta$. Notice that the COF of the min-max MTSP solution, $\mathrm{COF}(\mathcal F_{\infty})$, is a fixed value, whereas for the  $\varepsilon$-F-MTSP (resp. $\Delta$-F-MTSP) the COF monotonically increases (resp. decreases) with $\varepsilon$ (resp. $\Delta$) values. This monotonicity property is theoretically guaranteed by  Corollary \ref{cor:cof} and Proposition \ref{prop:delta-monotonicity}, respectively. Figure \ref{fig:cof} is an important tool for practitioners to understand the trade-off between efficiency, measured by total travel distance, and fairness in the distribution of individual tour lengths. One main advantage of the $\varepsilon$-F-MTSP and the $\Delta$-F-MTSP is that they provide a range of solutions with different COF values that a practitioner can choose from, as opposed to just one solution that has a maximum level of fairness provided by the min-max MTSP. We also note that the similarity between the two plots in Figures \ref{fig:pareto} and \ref{fig:cof}, aside from differences in scale, is not a coincidence. This is because the expression for COF in \eqref{eq:cof} and the sum of the tour lengths are related by an affine transformation. 
\begin{figure}[htbp]
    \centering
    \includegraphics[scale=1.0]{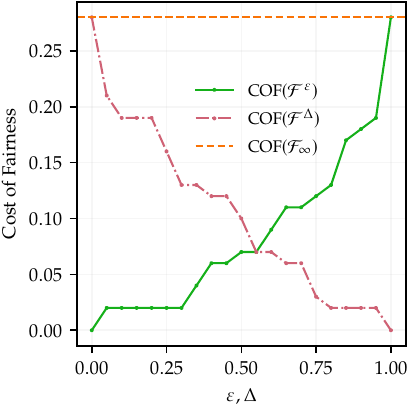}
    \caption{Cost of fairness of the optimal solutions of $\varepsilon$-F-MTSP and $\Delta$-F-MTSP as we vary $\varepsilon$ and $\Delta$ from 0 to 1 for the eil51 instance with $m=5$.}
    \label{fig:cof}
\end{figure}

\subsection{Direct comparison of F-MTSP variants with min-max MTSP} 
To perform a direct comparison, we solve the min-max MTSP $\mathcal F_{\infty}$, compute $\varepsilon(\mathcal F_{\infty})$ and $\Delta(\mathcal F_{\infty})$ using the optimal solution of the min-max MTSP, i.e., $\bm l^*(\mathcal F_{\infty})$, as 
\begin{gather}
    \varepsilon(\mathcal F_{\infty}) \triangleq \varepsilon\mathrm{FI}(\bm l^*(\mathcal F_{\infty})) \label{eq:eps-min-max} \\ 
    \Delta(\mathcal F_{\infty}) \triangleq \mathrm{GC}(\bm l^*(\mathcal F_{\infty})) \label{eq:delta-min-max}
\end{gather}
and solve the $\varepsilon$-F-MTSP and $\Delta$-F-MTSP using $\varepsilon = \varepsilon(\mathcal F_{\infty})$ and $\Delta = \Delta(\mathcal F_{\infty})$, respectively. Essentially, we calculate the level of fairness in the optimal solution to the min-max MTSP, $\bm l^*(\mathcal F_{\infty})$, using the $\varepsilon$-fair index in \eqref{eq:efi} and Gini coefficient in \eqref{eq:gi} and solve $\mathcal F^{\varepsilon}$ and $\mathcal F^{\Delta}$ with the calculated values. \textcolor{black}{Table \ref{tab:min-max-fair} reports both the maximum tour cost and the sum of all tour costs obtained by the above procedure for different instances. The instances shown correspond to those for which the min–max MTSP was solved to optimality within the one-hour time limit. Since the min-max formulation explicitly minimizes the maximum tour cost, this metric serves as the primary measure of fairness, while the sum of tour costs provides insight into overall efficiency. We note that in Table \ref{tab:min-max-fair}, the $\Delta$-F-MTSP variant timed out for the $\Delta$ value of $0.00384$ for the instance ``att48-4'', indicating that achieving the same level of fairness as the min-max MTSP can still be computationally challenging using the proposed fair variants. Nevertheless, solutions with slightly relaxed fairness levels can be obtained with significantly lower computational effort, as discussed in Section \ref{subsec:time}. An important takeaway from Table \ref{tab:min-max-fair} is that the F-MTSP variants can attain comparable maximum tour costs (i.e., similar fairness levels) while yielding a lower total tour cost, demonstrating a more favorable trade-off between fairness and overall efficiency compared to the min-max MTSP.}

\setlength{\tabcolsep}{10pt}
\begin{table}[htbp]
    \centering
    \caption{Objective value comparison between min-max and fair variants of the MTSP.}
    \label{tab:min-max-fair}
    \addtolength{\tabcolsep}{-0.5em}
    \begin{tabular}{lrrrrrrrr}
    \toprule
    \multirow{2}{*}{name} & \multirow{2}{*}{$\|\bm l^*(\mathcal F^{\infty})\|_1$} & \multirow{2}{*}{$\|\bm l^*(\mathcal F^{\infty})\|_\infty$} & \multicolumn{2}{c}{value of} & \multirow{2}{*}{$\|\bm l^*(\mathcal F^{\varepsilon})\|_1$} & \multirow{2}{*}{$\|\bm l^*(\mathcal F^{\Delta})\|_1$} & \multirow{2}{*}{$\|\bm l^*(\mathcal F^{\varepsilon})\|_\infty$} & \multirow{2}{*}{$\|\bm l^*(\mathcal F^{\Delta})\|_\infty$} \\
    \cmidrule{4-5}
    & & & $\varepsilon(\mathcal F^{\infty})$ & $\Delta(\mathcal F^{\infty})$ &  & & & \\
    \cmidrule{1-9}
    \begin{filecontents*}{tables/minmaxFair_final.csv}
Instance Name,Number of Vehicles,minmax cost,minmaxCOF,minmax eps,minmax delta,max_tour_min_max,epsFair cost,epsFair COF,max_tour_epsFair,deltaFair cost,deltaFair COF,max_tour_deltaFair
burma14,3,37.0,0.194,0.99312,0.05405,13.0,37.0,0.194,13.0,37.0,0.194,13.0
burma14,4,48.0,0.548,1.0,0.0,12.0,48.0,0.548,12.0,48.0,0.548,12.0
burma14,5,54.0,0.742,0.97153,0.10185,12.0,49.0,0.581,12.0,50.0,0.613,12.0
gr17,3,3565.0,0.71,0.99407,0.0533,1260.0,3451.0,0.655,1266.0,3439.0,0.649,1266.0
gr17,4,4716.0,1.262,0.99791,0.03364,1260.0,4670.0,1.24,1260.0,4584.0,1.199,1260.0
gr17,5,5217.0,1.502,0.88655,0.19293,1260.0,3995.0,0.916,1266.0,4074.0,0.954,1435.0
gr21,3,4304.0,0.59,0.99995,0.00465,1442.0,4304.0,0.59,1442.0,4304.0,0.59,1442.0
gr21,4,5427.0,1.005,0.9993,0.01726,1384.0,5315.0,0.963,1389.0,5301.0,0.958,1391.0
gr21,5,6367.0,1.352,0.99965,0.01241,1294.0,6263.0,1.314,1296.0,6282.0,1.321,1296.0
gr24,3,1928.0,0.516,1.0,0.00052,643.0,1928.0,0.516,649.0,1928.0,0.516,643.0
gr24,4,2425.0,0.906,0.99996,0.00399,609.0,2424.0,0.906,615.0,2425.0,0.906,609.0
gr24,5,2867.0,1.254,0.99938,0.01622,585.0,2836.0,1.23,594.0,2847.0,1.238,600.0
fri26,3,1433.0,0.529,1.0,0.0007,478.0,1424.0,0.52,478.0,1432.0,0.528,478.0
fri26,4,1750.0,0.868,0.99996,0.00457,440.0,1735.0,0.852,440.0,1746.0,0.863,440.0
fri26,5,2129.0,1.272,0.99919,0.01973,440.0,2071.0,1.21,440.0,2065.0,1.204,440.0
bayg29,3,1928.0,0.198,0.99786,0.0306,663.0,1909.0,0.186,675.0,1912.0,0.188,675.0
bayg29,4,2254.0,0.4,0.99958,0.01331,572.0,2213.0,0.375,574.0,2223.0,0.381,575.0
bayg29,5,2699.0,0.676,0.99987,0.00797,547.0,2684.0,0.667,547.0,2695.0,0.674,548.0
bays29,3,10569.0,0.159,0.99993,0.00596,3548.0,10555.0,0.158,3551.0,10564.0,0.159,3551.0
bays29,4,11305.0,0.24,0.99983,0.00858,2859.0,11305.0,0.24,2859.0,11305.0,0.24,2859.0
gr48,3,6016.0,0.192,0.99998,0.00349,2015.0,6016.0,0.192,2015.0,6016.0,0.192,2015.0
att48,3,38865.0,0.159,0.99971,0.0122,13147.0,38783.0,0.157,13212.0,38751.0,0.156,13212.0
att48,4,44766.0,0.335,0.99997,0.00384,11276.0,44509.0,0.328,11311.0,-,-,-
eil51,3,472.0,0.098,0.99996,0.00424,158.0,470.0,0.093,158.0,470.0,0.093,158.0
eil51,4,494.0,0.149,0.99998,0.0027,124.0,492.0,0.144,125.0,494.0,0.149,124.0
eil76,3,574.0,0.065,0.99997,0.00348,192.0,574.0,0.065,192.0,574.0,0.065,192.0
A-n33-k5,5,765.0,0.755,0.99901,0.0183,156.0,744.0,0.706,157.0,749.0,0.718,157.0
A-n34-k5,5,811.0,0.662,0.99991,0.00678,164.0,806.0,0.652,164.0,806.0,0.652,164.0
E-n13-k4,4,290.0,0.895,0.99919,0.01839,74.0,290.0,0.895,74.0,290.0,0.895,74.0
E-n22-k4,4,428.0,0.54,0.99895,0.01869,109.0,417.0,0.5,109.0,419.0,0.507,109.0
E-n23-k3,3,630.0,0.34,1.0,0.0,210.0,630.0,0.34,210.0,642.0,0.366,214.0
E-n30-k3,3,494.0,0.297,0.99999,0.00202,165.0,491.0,0.289,165.0,493.0,0.294,165.0
\end{filecontents*}
\csvreader[head to column names, late after line=\\]{tables/minmaxFair_final.csv}{}
    {\csvcoli-\csvcolii & \csvcoliii & \csvcolvii & \csvcolv & \csvcolvi & \csvcolviii & \dashToTO{\csvcolxi} & \csvcolx & \dashToTO{\csvcolxiii} }
    \bottomrule
\end{tabular}
\end{table}

\begin{figure}[htbp]
    \centering
    \includegraphics[scale=1.0]{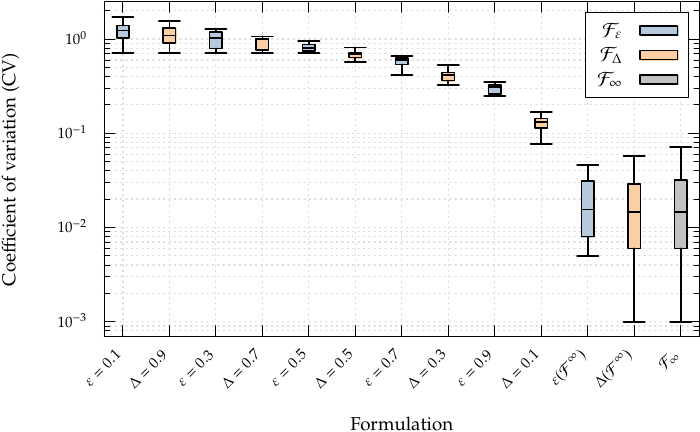}
    \caption{Boxplot of coefficient of variation (CV) values of tour lengths across formulations and parameter settings for all the instances in Table \ref{tab:min-max-fair}.}
    \label{fig:cv-boxplot}
\end{figure}
\textcolor{black}{To further quantify the fairness--efficiency trade-off, Figure~\ref{fig:cv-boxplot} reports the distribution of the coefficient of variation (CV) of tour lengths across formulations on a logarithmic scale. The results clearly show that $\mathcal{F}_{\infty}$ (the min-max MTSP) achieves the lowest CV values, confirming its role in enforcing the strongest fairness. Both $\varepsilon$-F-MTSP and $\Delta$-F-MTSP exhibit a consistent and monotonic trend: increasing $\varepsilon$ (and decreasing $\Delta$) systematically reduces the CV, yielding progressively more balanced tours. For moderate parameter values, these formulations already attain CV levels close to $\mathcal{F}_{\infty}$. Moreover, when evaluated at $\varepsilon(\mathcal{F}_{\infty})$ and $\Delta(\mathcal{F}_{\infty})$, both variants closely match the CV of $\mathcal{F}_{\infty}$, demonstrating that the same level of fairness can be recovered without incurring the full computational burden of the min-max MTSP formulation. Overall, Figure~\ref{fig:cv-boxplot} demonstrates that the proposed fair variants provide a precise and effective mechanism for achieving near-min-max fairness while improving computational efficiency.}

\subsection{A practical use case - electric vehicle fleet management} \label{subsec:real}
This section demonstrates the practical utility of the F-MTSP formulation. Suppose we have a fleet of 4 electric vehicles stationed at a depot for package deliveries in a Seattle suburb. The objective is to find routes for each vehicle that minimize the total travel distance while ensuring a fair distribution of travel distances among vehicles. Enforcing that the distribution of vehicles' travel distances is fair is useful in this context, as it ensures that all vehicles have equitable battery usage. This problem can be directly modeled as an MTSP. We now qualitatively compare the optimal solutions provided by the different variants of the MTSP in Figure \ref{fig:seattle}. 

\begin{figure*}[htbp]
\centering
  \begin{subfigure}[t]{.3\linewidth}
    \centering\includegraphics[scale=1]{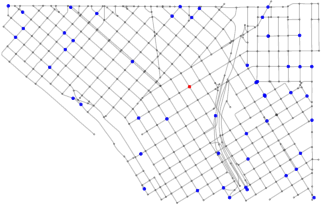}
    \caption{Seattle instance}
    \label{fig:seattle-target}
  \end{subfigure} \hspace{3ex}
  \begin{subfigure}[t]{.3\linewidth}
    \centering\includegraphics[scale=1]{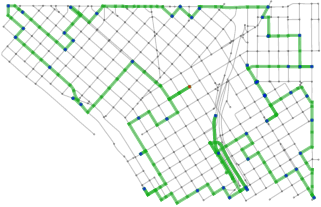}
    \caption{Min-sum MTSP (time: 17.24 Section)}
    \label{fig:seattle-min}
  \end{subfigure}\hspace{3ex}
  \begin{subfigure}[t]{.3\linewidth}
    \centering\includegraphics[scale=1]{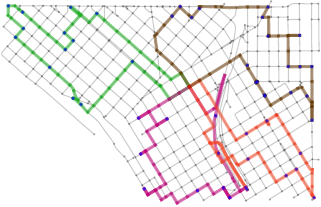}
    \caption{Min-max MTSP (time: 633.83 Section)}
    \label{fig:seattle-minmax}
  \end{subfigure} \\\vspace{2ex}
  \begin{subfigure}[t]{.3\linewidth}
    \centering\includegraphics[scale=1]{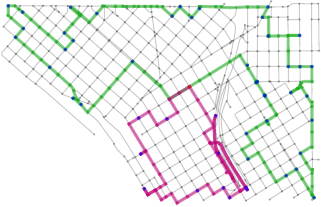}
    \caption{$\mathcal F^{\Delta}$ with $\Delta = 0.9$ (time: 30.3 Section)}
    \label{fig:seattle-delta-9}
  \end{subfigure} \hspace{3ex}
  \begin{subfigure}[t]{.3\linewidth}
    \centering\includegraphics[scale=1]{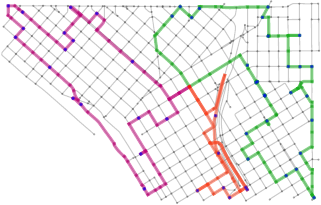}
    \caption{$\mathcal F^{\Delta}$ with $\Delta = 0.5$ (time: 50.32 Section)}
    \label{fig:seattle-delta-5}
  \end{subfigure}\hspace{3ex}
  \begin{subfigure}[t]{.3\linewidth}
    \centering\includegraphics[scale=1]{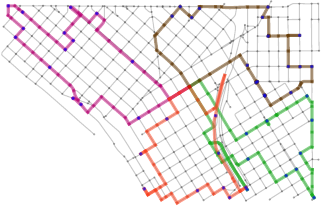}
    \caption{$\mathcal F^{\Delta}$ with $\Delta = 0.1$ (time: 153.14 Section)}
    \label{fig:seattle-delta-1}
  \end{subfigure} \\\vspace{2ex}
  \begin{subfigure}[t]{.3\linewidth}
    \centering\includegraphics[scale=1]{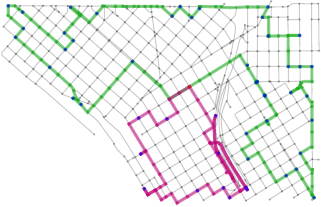}
    \caption{$\mathcal F^{\varepsilon}$ with $\varepsilon = 0.1$ (time: 34.32 Section)}
    \label{fig:seattle-eps-1}
  \end{subfigure} \hspace{3ex}
  \begin{subfigure}[t]{.3\linewidth}
    \centering\includegraphics[scale=1]{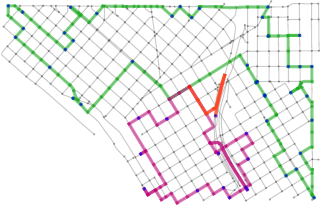}
    \caption{$\mathcal F^{\varepsilon}$ with $\varepsilon = 0.5$ (time: 91.2 Section)}
    \label{fig:seattle-eps-5}
  \end{subfigure}\hspace{3ex}
  \begin{subfigure}[t]{.3\linewidth}
    \centering\includegraphics[scale=1]{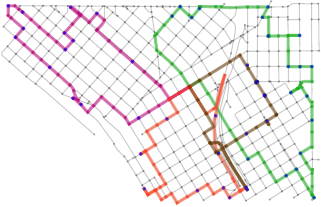}
    \caption{$\mathcal F^{\varepsilon}$ with $\varepsilon = 0.9$ (time: 271.52 Section)}
    \label{fig:seattle-eps-9}
  \end{subfigure}
  \caption{(a) shows the graph of the ``Seattle'' instance with the red dot representing the depot and the blue dots representing the different targets. $4$ electric vehicles are initially stationed at the depot. The path for each vehicle obtained by solving the different variants of the MTSP is shown in different colors in (b)--(i). (b) and (c) show the optimal min-sum and min-max MTSP tours, respectively. The second and third rows show optimal solutions for the $\Delta$-F-MTSP and $\varepsilon$-F-MTSP, respectively, with different values of $\Delta$ and $\varepsilon$. The time taken to compute the optimal solution in all the illustrations is provided in parentheses.}
  \label{fig:seattle}
\end{figure*}

Note that the optimal solution to the min-sum MTSP (Figure \ref{fig:seattle-min}) uses only one vehicle to perform all deliveries, as it imposes no fairness constraints on the distribution of each vehicle's travel distance. From a pure minimize travel distance standpoint, it is always better to only use one vehicle to visit all the targets as long as the triangle inequality is satisfied, and this is exactly what is observed in the optimal min-sum MTSP solution. The optimal min-max MTSP solution uses all four vehicles and qualitatively has a fair distribution of the vehicles' travel distances (Figure \ref{fig:seattle-minmax}). Solutions with a similar distribution of tour lengths are obtained by solving the $\Delta$-F-MTSP with $\Delta = 0.1$ (Figure \ref{fig:seattle-delta-1}) and $\varepsilon$-F-MTSP with $\varepsilon = 0.9$ (Figure \ref{fig:seattle-eps-9}) with much lower computation times. Finally, as we either increase $\varepsilon$ (Figures \ref{fig:seattle-eps-1}--\ref{fig:seattle-eps-9}) or decrease $\Delta$ (Figures \ref{fig:seattle-delta-9}--\ref{fig:seattle-delta-1}), the solutions qualitatively become fairer and the computation time to find the fairer solutions increases.

\section{Conclusions and way forward} \label{sec:conclusion}
\textcolor{black}{This work proposes different variants of the MTSP to incorporate fairness in the distribution of tour lengths. From a modeling perspective, these variants can be categorized into two classes. The first class consists of objective-based formulations, namely the min-max MTSP and the $p$-norm MTSP, which directly modify the objective function to promote fairness. The second class comprises constraint-based formulations, namely the $\varepsilon$-fair MTSP and the $\Delta$-fair MTSP, which retain the classical min-sum objective while enforcing fairness via additional constraints. The first two formulations exist in the literature, while the latter two are the novel contributions of this work. The work also develops a custom branch-and-cut algorithm to solve all the variants. The major takeaways of this work can be summarised as follows:}
\begin{enumerate}[label=(\roman*)]
    \item The min-max MTSP can achieve fairness in the distribution of tour lengths at the cost of high computation time. This is well known in the literature, and the results of the computational experiments are consistent with this observation.
    \item \textcolor{black}{Among the objective-based formulations, the $p$-norm MTSP with $p=2$ provides a practical alternative to the min-max MTSP in terms of both computation time and fairness. Among the constraint-based formulations, the $\varepsilon$-F-MTSP with a large $\varepsilon$ or the $\Delta$-F-MTSP with a small $\Delta$ achieve comparable fairness while incurring significantly lower computational effort. Furthermore, the min-max MTSP is equivalent to the $p$-norm MTSP when $p = \infty$.}
    \item All three variants, namely the $p$-norm MTSP, the $\varepsilon$-F-MTSP, and the $\Delta$-F-MTSP, can result in a family of solutions with different levels of the trade-off between the sum of tour lengths and fairness in the distribution of tour lengths when the values of $p$, $\varepsilon$, and $\Delta$ are varied in their respective domains, with the $\varepsilon$-F-MTSP and $\Delta$-F-MTSP having a significant edge in terms of computation time for larger test cases. Additionally, the $\varepsilon$-F-MTSP and the $\Delta$-F-MTSP have theoretical properties that conform with our intuitive understanding of the inherent trade-off between fairness and efficiency, while the $p$-norm does not have such properties.
    \item The algorithms developed for the $\varepsilon$-F-MTSP and the $\Delta$-F-MTSP can be used to compute the Pareto front of the two bi-objective MTSP formulations $\mathcal{F}^{\varepsilon}_{\mathrm{bi-obj}}$ and $\mathcal{F}^{\Delta}_{\mathrm{bi-obj}}$ with both the efficiency and fairness objectives. 
    \item Finally, empirical results suggest that the computation time of any of the fair variants of the MTSP proposed in this work increases as one tries to make the distributions of tour lengths fairer. 
\end{enumerate}

Future work may focus on developing fast heuristics and approximation algorithms for the fair variants of the MTSP. Furthermore, leveraging a general model that enforces a set of decision variable values to be fair across application domains such as scheduling, supply chain management, robotics, and energy systems will be an exciting research direction. \\

\noindent \textbf{Disclosure of interest} - The authors declare that they have no competing interests. \\

\noindent \textbf{Funding} - No funding was received to carry out this work. \\

\noindent \textbf{Data availability statement} - The data and source code to reproduce all the results are made available in the GitHub repository \url{https://github.com/kaarthiksundar/FairMTSP}. 

\printbibliography

\appendix 
\section{Comprehensive runtime results} \label{app:runtime} This section presents the comprehensive instance-by-instance run times for all the formulations proposed in this work.
\begin{table}[htbp]
    \centering
    \caption{TSPLIB instances -- computation time in seconds for min-sum MTSP, min-max MTSP, and $p$-norm MTSP. For all the instances that timed out, the numbers in parentheses show the optimality gap at the computation time limit of one hour. }
    \label{tab:pNorm}
    \begin{tabular}{lrrrrrr}
    \toprule
    \multirow{2}{*}{name} & \multirow{2}{*}{$\mathcal F_1$} & \multirow{2}{*}{$\mathcal F_{\infty}$} & \multicolumn{4}{c}{$\mathcal F_p$}\\
    \cmidrule{4-7}
    & & & $p=2$ & $p=3$ & $p=5$ & $p=10$ \\
    \cmidrule{1-7}
    \begin{filecontents*}{tables/p-norm-tsp.csv}
Instance Name,Number of Vehicles,min (sec),minmax ,pNorm 2,pNorm 3,pNorm 5,pNorm 10
burma14,3,0.33,0.71,0.64,0.67,0.76,0.48
burma14,4,0.45,0.88,1.48,2.02,1.2,2.34
burma14,5,1.0,2.0,2.65,3.13,3.79,2.79
gr17,3,0.52,1.46,1.42,1.51,1.45,2.74
gr17,4,0.41,3.48,2.92,3.13,3.86,8.16
gr17,5,0.59,12.03,3.31,7.2,6.33,6.36
gr21,3,0.35,6.58,2.23,5.63,4.3,5.02
gr21,4,0.45,24.57,5.67,7.37,13.42,13.42
gr21,5,0.26,139.75,11.49,16.83,17.99,27.14
gr24,3,0.62,16.06,7.44,5.45,11.29,11.77
gr24,4,1.4,271.01,36.25,54.31,81.87,220.41
gr24,5,1.43,1933.85,122.84,280.37,678.98,842.29
fri26,3,0.78,9.04,8.32,8.08,10.78,6.77
fri26,4,2.14,30.78,22.05,37.61,29.01,18.03
fri26,5,1.46,58.48,60.92,82.61,55.27,60.09
bayg29,3,1.18,18.53,7.99,7.3,6.52,9.04
bayg29,4,2.15,84.64,17.08,25.58,63.57,48.65
bayg29,5,2.62,2657.92,45.74,140.35,385.18,240.71
bays29,3,1.03,15.38,7.52,7.96,8.0,8.21
bays29,4,3.61,45.18,14.63,15.05,14.84,28.17
bays29,5,5.34,TO (0.34),66.13,129.02,75.16,163.8
dantzig42,3,3.21,TO (0.03),173.93,832.78,2093.28,TO (0.02)
dantzig42,4,4.76,TO (0.18),825.83,TO (0.11),TO (0.09),TO (0.11)
dantzig42,5,7.04,TO (0.39),TO (0.03),TO (0.39),TO (0.27),TO (0.52)
gr48,3,9.77,327.1,151.46,104.25,150.43,324.16
gr48,4,8.95,TO (0.06),2344.0,TO (0.06),2909.32,TO (0.04)
gr48,5,21.65,TO (0.39),TO (0.2),TO (0.25),TO (0.2),TO (0.16)
att48,3,4.97,57.41,32.91,28.05,37.79,24.65
att48,4,6.83,675.81,357.23,106.62,220.68,199.84
att48,5,12.69,TO (0.05),2447.5,1870.32,1557.95,2892.54
eil51,3,5.22,88.34,35.78,60.59,45.56,42.71
eil51,4,20.51,675.6,203.81,228.53,237.73,200.76
eil51,5,30.47,TO (0.07),876.74,1317.25,858.31,2710.49
eil76,3,9.04,1064.41,1090.45,922.68,440.79,559.85
eil76,4,16.73,TO (0.48),TO (0.02),TO (0.02),TO (0.01),TO (0.05)
eil76,5,63.47,TO (0.37),TO (0.3),TO (0.29),TO (0.3),TO (0.24)
\end{filecontents*}
\csvreader[head to column names, late after line=\\]{tables/p-norm-tsp.csv}{}
    {\csvcoli-\csvcolii & \dashToTO{\csvcoliii} & \dashToTO{\csvcoliv} & \dashToTO{\csvcolv} & \dashToTO{\csvcolvi} & \dashToTO{\csvcolvii} & \dashToTO{\csvcolviii}}
    \bottomrule
\end{tabular}
\end{table}
\begin{table}[htbp]
    \centering
    \caption{TSPLIB instances -- computation time in seconds for min-sum MTSP and $\varepsilon$-F-MTSP. For all the instances that timed out, the numbers in parentheses show the optimality gap at the computation time limit of one hour. }
    \label{tab:epsFair}
    \begin{tabular}{lrrrrrr}
    \toprule
    \multirow{2}{*}{name} & \multirow{2}{*}{$\mathcal F_1$} & \multicolumn{4}{c}{$\mathcal F^{\varepsilon}$}\\
    \cmidrule{3-7}
    & & $\varepsilon = 0.1$ & $\varepsilon = 0.3$ & $\varepsilon = 0.5$ & $\varepsilon = 0.7$ & $\varepsilon = 0.9$\\
    \cmidrule{1-7}
    \begin{filecontents*}{tables/eps-fair-tsp.csv}
Instance Name,Number of Vehicles,min (sec),minmax,epsFair 0.1,epsFair 0.3,epsFair 0.5,epsFair 0.7,epsFair 0.9
burma14,3,0.33,0.71,0.37,0.57,0.46,0.45,0.47
burma14,4,0.45,0.88,0.69,0.77,1.0,0.79,1.23
burma14,5,1.0,2.0,0.64,0.45,2.31,2.98,3.71
gr17,3,0.52,1.46,0.93,0.61,2.26,2.17,3.55
gr17,4,0.41,3.48,0.97,1.5,2.18,5.05,5.43
gr17,5,0.59,12.03,2.05,2.86,4.9,7.23,9.56
gr21,3,0.35,6.58,1.91,3.19,2.32,2.59,3.88
gr21,4,0.45,24.57,4.31,2.59,4.63,9.81,14.63
gr21,5,0.26,139.75,2.62,3.77,14.66,13.71,34.46
gr24,3,0.62,16.06,3.66,3.84,4.06,12.98,9.04
gr24,4,1.4,271.01,9.02,5.75,22.34,47.14,30.07
gr24,5,1.43,1933.85,12.33,15.5,39.28,111.62,210.76
fri26,3,0.78,9.04,3.77,5.28,3.54,10.19,5.58
fri26,4,2.14,30.78,4.87,8.61,18.61,15.97,15.62
fri26,5,1.46,58.48,5.01,6.42,23.72,95.94,115.06
bayg29,3,1.18,18.53,5.78,4.28,7.95,10.41,7.71
bayg29,4,2.15,84.64,5.28,5.17,12.24,12.81,15.14
bayg29,5,2.62,2657.92,10.07,11.15,23.04,45.38,115.52
bays29,3,1.03,15.38,4.3,6.01,6.66,8.24,136.25
bays29,4,3.61,45.18,5.69,7.29,9.6,16.35,13.9
bays29,5,5.34,TO (0.34),10.41,11.68,18.27,46.0,62.79
dantzig42,3,3.21,TO (0.03),11.2,23.47,42.1,69.96,255.28
dantzig42,4,4.76,TO (0.18),23.24,31.72,221.55,1098.86,TO (0.02)
dantzig42,5,7.04,TO (0.39),90.47,105.66,711.65,TO (0.1),TO (0.1)
gr48,3,9.77,327.1,19.54,21.77,17.11,76.2,75.47
gr48,4,8.95,TO (0.06),40.8,35.26,151.69,246.84,1002.27
gr48,5,21.65,TO (0.39),41.66,37.68,509.35,3262.45,TO (0.16)
att48,3,4.97,57.41,19.29,16.21,11.55,35.89,27.62
att48,4,6.83,675.81,28.84,26.17,50.48,66.43,136.1
att48,5,12.69,TO (0.05),44.51,51.07,68.22,280.67,1019.86
eil51,3,5.22,88.34,19.14,25.5,19.88,60.83,23.04
eil51,4,20.51,675.6,35.47,33.53,120.56,88.04,101.89
eil51,5,30.47,TO (0.07),54.35,55.34,128.35,342.25,1014.69
eil76,3,9.04,1064.41,126.62,101.82,96.3,444.11,415.54
eil76,4,16.73,TO (0.48),228.92,456.31,TO (0.33),1698.34,TO (0.03)
eil76,5,63.47,TO (0.37),402.87,208.48,TO (0.03),TO (0.07),TO (0.54)
\end{filecontents*}
\csvreader[head to column names, late after line=\\]{tables/eps-fair-tsp.csv}{}
    {\csvcoli-\csvcolii & \dashToTO{\csvcoliii} & \dashToTO{\csvcolv} & \dashToTO{\csvcolvi} & \dashToTO{\csvcolvii} & \dashToTO{\csvcolviii} & \dashToTO{\csvcolix}}
    \bottomrule
\end{tabular}
\end{table}
\begin{table}[htbp]
    \centering
    \caption{TSPLIB instances -- computation time in seconds for min-sum MTSP and $\Delta$-F-MTSP. For all the instances that timed out, the numbers in parentheses show the optimality gap at the computation time limit of one hour. }
    \label{tab:deltaFair}
    \begin{tabular}{lrrrrrr}
    \toprule
    \multirow{2}{*}{name} & \multirow{2}{*}{$\mathcal F_1$} & \multicolumn{4}{c}{$\mathcal F^{\Delta}$}\\
    \cmidrule{3-7}
    & & $\Delta = 0.1$ & $\Delta = 0.3$ & $\Delta = 0.5$ & $\Delta = 0.7$ & $\Delta = 0.9$\\
    \cmidrule{1-7}
    \begin{filecontents*}{tables/delta-fair-tsp.csv}
Instance Name,Number of Vehicles,min (sec),minmax,deltaFair 0.1,deltaFair 0.3,deltaFair 0.5,deltaFair 0.7,deltaFair 0.9
burma14,3,0.33,0.71,0.46,0.59,0.47,0.41,0.5
burma14,4,0.45,0.88,1.69,1.81,0.91,0.69,0.38
burma14,5,1.0,2.0,3.06,2.82,3.28,1.52,0.75
gr17,3,0.52,1.46,2.16,1.33,1.67,0.63,0.89
gr17,4,0.41,3.48,7.58,4.59,3.23,3.42,2.62
gr17,5,0.59,12.03,43.57,7.55,4.95,2.89,3.75
gr21,3,0.35,6.58,4.25,2.55,2.71,2.47,1.71
gr21,4,0.45,24.57,7.9,7.5,7.29,4.01,3.2
gr21,5,0.26,139.75,37.54,39.9,16.53,7.65,4.87
gr24,3,0.62,16.06,9.33,6.92,6.53,4.44,2.52
gr24,4,1.4,271.01,57.1,67.23,63.35,8.5,7.67
gr24,5,1.43,1933.85,316.64,194.65,169.37,45.92,9.71
fri26,3,0.78,9.04,8.52,7.35,11.12,3.75,4.08
fri26,4,2.14,30.78,22.79,46.65,35.7,15.91,7.28
fri26,5,1.46,58.48,45.56,65.2,183.7,24.87,5.91
bayg29,3,1.18,18.53,8.26,12.45,7.71,5.86,6.6
bayg29,4,2.15,84.64,22.35,22.52,11.69,10.49,5.4
bayg29,5,2.62,2657.92,99.79,104.61,36.91,17.26,8.7
bays29,3,1.03,15.38,224.72,8.64,17.61,6.96,5.13
bays29,4,3.61,45.18,19.25,16.37,13.8,7.28,7.83
bays29,5,5.34,TO (0.34),65.67,28.42,29.4,23.01,10.83
dantzig42,3,3.21,TO (0.03),141.58,89.93,71.95,30.45,12.69
dantzig42,4,4.76,TO (0.18),TO (0.07),1174.31,813.06,85.52,37.42
dantzig42,5,7.04,TO (0.39),TO (0.25),TO (0.08),TO (0.03),768.37,55.77
gr48,3,9.77,327.1,147.64,215.52,49.14,18.43,17.42
gr48,4,8.95,TO (0.06),2790.03,1062.36,167.17,24.11,31.32
gr48,5,21.65,TO (0.39),TO (0.2),TO (0.04),TO (0.01),535.01,56.6
att48,3,4.97,57.41,46.39,18.91,108.43,18.03,13.97
att48,4,6.83,675.81,279.58,71.01,59.48,24.6,27.66
att48,5,12.69,TO (0.05),TO (0.02),841.75,398.19,82.2,41.5
eil51,3,5.22,88.34,40.96,32.75,TO (0.15),17.4,8.49
eil51,4,20.51,675.6,93.72,242.88,73.94,36.6,40.8
eil51,5,30.47,TO (0.07),707.15,365.4,609.62,251.18,62.68
eil76,3,9.04,1064.41,TO (0.83),306.24,405.47,122.9,TO (0.12)
eil76,4,16.73,TO (0.48),TO (0.05),TO (0.03),932.19,316.48,218.01
eil76,5,63.47,TO (0.37),TO (0.76),TO (0.11),TO (0.07),TO (0.02),TO (0.14)
\end{filecontents*}
\csvreader[head to column names, late after line=\\]{tables/delta-fair-tsp.csv}{}
    {\csvcoli-\csvcolii & \dashToTO{\csvcoliii} & \dashToTO{\csvcolv} & \dashToTO{\csvcolvi} & \dashToTO{\csvcolvii} & \dashToTO{\csvcolviii} & \dashToTO{\csvcolix}}
    \bottomrule
\end{tabular}
\end{table}
\begin{table}[htbp]
    \centering
    \caption{CVRPLIB instances -- computation time in seconds for min-sum MTSP, min-max MTSP, and $p$-norm MTSP. For all the instances that timed out, the numbers in parentheses show the optimality gap at the computation time limit of one hour. }
    \label{tab:pNorm_vrp}
    \begin{tabular}{lrrrrrr}
    \toprule
    \multirow{2}{*}{name} & \multirow{2}{*}{$\mathcal F_1$} & \multirow{2}{*}{$\mathcal F_{\infty}$} & \multicolumn{4}{c}{$\mathcal F_p$}\\
    \cmidrule{4-7}
    & & & $p=2$ & $p=3$ & $p=5$ & $p=10$ \\
    \cmidrule{1-7}
    \begin{filecontents*}{tables/p-norm-vrp.csv}
Instance Name,Number of Vehicles,min (sec),minmax ,pNorm 2,pNorm 3,pNorm 5,pNorm 10
A-n32-k5,5,2.1,TO (0.03),344.18,313.24,445.24,541.0
A-n33-k5,5,4.57,2515.92,150.36,171.57,270.5,505.83
A-n33-k6,6,5.03,TO (0.09),696.25,663.47,1504.62,2141.66
A-n34-k5,5,6.52,653.54,262.12,126.28,236.22,237.7
A-n36-k5,5,2.47,TO (0.12),1017.88,2623.83,TO (0.05),TO (0.08)
A-n37-k5,5,2.82,TO (0.02),462.21,347.5,394.69,750.2
A-n37-k6,6,4.21,TO (0.37),TO (0.03),TO (0.11),TO (0.29),TO (0.07)
A-n38-k5,5,10.67,TO (0.01),335.3,276.52,1460.88,790.71
A-n39-k5,5,5.16,TO (0.24),1183.26,TO (0.02),TO (0.1),TO (0.08)
A-n39-k6,6,16.21,TO (0.26),TO (0.09),TO (0.06),TO (0.15),TO (0.1)
A-n44-k6,6,27.24,TO (0.47),TO (0.19),TO (0.33),TO (0.34),TO (0.42)
A-n45-k6,6,13.58,TO (0.4),2773.11,TO (0.06),TO (0.21),TO (0.35)
A-n45-k7,7,41.15,TO (0.68),TO (0.49),TO (0.69),TO (0.6),TO (0.66)
A-n46-k7,7,26.62,TO (0.61),TO (0.5),TO (0.44),TO (0.62),TO (0.56)
A-n48-k7,7,51.32,TO (0.71),TO (0.74),TO (0.73),TO (0.61),TO (0.77)
E-n13-k4,4,0.19,4.48,1.87,4.3,2.58,3.27
E-n22-k4,4,0.98,24.2,8.41,9.71,8.92,13.92
E-n23-k3,3,0.67,5.14,3.15,5.6,5.67,3.98
E-n30-k3,3,0.92,10.07,5.18,9.3,5.21,9.24
E-n31-k7,7,3.19,TO (0.64),TO (0.22),TO (0.15),TO (0.3),TO (0.3)
E-n33-k4,4,3.67,TO (0.02),48.7,125.38,225.38,905.19
E-n51-k5,5,14.76,TO (0.13),1818.56,2254.99,TO (0.03),TO (0.03)
\end{filecontents*}
\csvreader[head to column names, late after line=\\]{tables/p-norm-vrp.csv}{}
    {\csvcoli-\csvcolii & \dashToTO{\csvcoliii} & \dashToTO{\csvcoliv} & \dashToTO{\csvcolv} & \dashToTO{\csvcolvi} & \dashToTO{\csvcolvii} & \dashToTO{\csvcolviii}}
    \bottomrule
\end{tabular}
\end{table}
\begin{table}[htbp]
    \centering
    \caption{CVRPLIB instances -- computation time in seconds for min-sum MTSP and $\varepsilon$-F-MTSP. For all the instances that timed out, the numbers in parentheses show the optimality gap at the computation time limit of one hour. }
    \label{tab:epsFair_vrp}
    \begin{tabular}{lrrrrrr}
    \toprule
    \multirow{2}{*}{name} & \multirow{2}{*}{$\mathcal F_1$} & \multicolumn{4}{c}{$\mathcal F^{\varepsilon}$}\\
    \cmidrule{3-7}
    & & $\varepsilon = 0.1$ & $\varepsilon = 0.3$ & $\varepsilon = 0.5$ & $\varepsilon = 0.7$ & $\varepsilon = 0.9$\\
    \cmidrule{1-7}
    \begin{filecontents*}{tables/eps-fair-vrp.csv}
Instance Name,Number of Vehicles,min (sec),minmax,epsFair 0.1,epsFair 0.3,epsFair 0.5,epsFair 0.7,epsFair 0.9
A-n32-k5,5,2.1,TO (0.03),11.8,21.23,50.82,131.44,484.41
A-n33-k5,5,4.57,2515.92,7.53,16.15,21.87,53.65,157.34
A-n33-k6,6,5.03,TO (0.09),24.66,37.49,71.63,406.97,468.65
A-n34-k5,5,6.52,653.54,17.44,18.5,35.95,98.5,91.04
A-n36-k5,5,2.47,TO (0.12),14.98,7.89,23.83,32.35,65.53
A-n37-k5,5,2.82,TO (0.02),14.22,27.27,40.27,135.23,226.38
A-n37-k6,6,4.21,TO (0.37),24.53,45.35,71.0,265.27,1271.46
A-n38-k5,5,10.67,TO (0.01),15.26,13.13,28.52,68.51,151.72
A-n39-k5,5,5.16,TO (0.24),22.19,24.3,56.9,151.55,2058.0
A-n39-k6,6,16.21,TO (0.26),69.51,60.58,312.4,1263.81,TO (0.04)
A-n44-k6,6,27.24,TO (0.47),63.99,162.15,621.07,2755.74,TO (0.06)
A-n45-k6,6,13.58,TO (0.4),64.72,74.68,124.8,440.86,3083.7
A-n45-k7,7,41.15,TO (0.68),96.55,521.14,TO (0.09),TO (0.35),TO (0.54)
A-n46-k7,7,26.62,TO (0.61),101.84,221.38,1461.56,TO (0.22),TO (0.45)
A-n48-k7,7,51.32,TO (0.71),133.89,1045.04,TO (0.16),TO (0.4),TO (0.51)
E-n13-k4,4,0.19,4.48,0.48,1.16,1.62,2.31,4.53
E-n22-k4,4,0.98,24.2,3.26,2.24,7.3,6.68,8.31
E-n23-k3,3,0.67,5.14,2.76,1.57,1.42,4.17,3.12
E-n30-k3,3,0.92,10.07,1.63,2.75,1.39,5.53,7.5
E-n31-k7,7,3.19,TO (0.64),19.13,113.2,356.37,TO (0.1),TO (0.24)
E-n33-k4,4,3.67,TO (0.02),10.92,10.48,69.86,145.29,126.66
E-n51-k5,5,14.76,TO (0.13),73.1,56.65,242.21,1047.32,TO (0.03)
\end{filecontents*}
\csvreader[head to column names, late after line=\\]{tables/eps-fair-vrp.csv}{}
    {\csvcoli-\csvcolii & \dashToTO{\csvcoliii} & \dashToTO{\csvcolv} & \dashToTO{\csvcolvi} & \dashToTO{\csvcolvii} & \dashToTO{\csvcolviii} & \dashToTO{\csvcolix}}
    \bottomrule
\end{tabular}
\end{table}
\begin{table}[htbp]
    \centering
    \caption{CVRPLIB instances -- computation time in seconds for min-sum MTSP and $\Delta$-F-MTSP. For all the instances that timed out, the numbers in parentheses show the optimality gap at the computation time limit of one hour. }
    \label{tab:deltaFair_vrp}
    \begin{tabular}{lrrrrrr}
    \toprule
    \multirow{2}{*}{name} & \multirow{2}{*}{$\mathcal F_1$} & \multicolumn{4}{c}{$\mathcal F^{\Delta}$}\\
    \cmidrule{3-7}
    & & $\Delta = 0.1$ & $\Delta = 0.3$ & $\Delta = 0.5$ & $\Delta = 0.7$ & $\Delta = 0.9$\\
    \cmidrule{1-7}
    \begin{filecontents*}{tables/delta-fair-vrp.csv}
Instance Name,Number of Vehicles,min (sec),minmax,deltaFair 0.1,deltaFair 0.3,deltaFair 0.5,deltaFair 0.7,deltaFair 0.9
A-n32-k5,5,2.1,TO (0.03),1535.17,193.66,113.34,46.26,24.83
A-n33-k5,5,4.57,2515.92,392.27,123.59,74.87,25.23,15.27
A-n33-k6,6,5.03,TO (0.09),1513.53,598.35,124.36,78.08,21.2
A-n34-k5,5,6.52,653.54,129.86,129.2,60.79,34.71,19.59
A-n36-k5,5,2.47,TO (0.12),3388.59,546.36,185.16,94.13,41.74
A-n37-k5,5,2.82,TO (0.02),224.49,130.59,74.14,47.24,9.27
A-n37-k6,6,4.21,TO (0.37),TO (0.28),TO (0.11),1891.62,277.52,109.72
A-n38-k5,5,10.67,TO (0.01),374.28,124.82,66.06,35.44,24.35
A-n39-k5,5,5.16,TO (0.24),TO (0.05),522.74,415.11,71.91,32.61
A-n39-k6,6,16.21,TO (0.26),TO (0.17),TO (0.08),1657.51,218.16,41.69
A-n44-k6,6,27.24,TO (0.47),TO (0.35),TO (0.16),2089.98,411.9,62.73
A-n45-k6,6,13.58,TO (0.4),TO (0.18),1709.74,373.62,184.38,58.04
A-n45-k7,7,41.15,TO (0.68),TO (0.61),TO (0.49),TO (0.16),TO (0.08),257.66
A-n46-k7,7,26.62,TO (0.61),TO (0.5),TO (0.38),TO (0.27),1012.82,101.57
A-n48-k7,7,51.32,TO (0.71),TO (0.63),TO (0.49),TO (0.38),TO (0.03),366.51
E-n13-k4,4,0.19,4.48,2.66,1.9,2.38,1.32,0.81
E-n22-k4,4,0.98,24.2,10.83,7.6,5.84,4.01,2.23
E-n23-k3,3,0.67,5.14,4.8,3.96,3.03,2.08,2.69
E-n30-k3,3,0.92,10.07,8.2,9.02,7.45,3.32,5.48
E-n31-k7,7,3.19,TO (0.64),TO (0.34),TO (0.16),TO (0.09),84.7,21.43
E-n33-k4,4,3.67,TO (0.02),32.15,20.62,14.11,9.91,6.75
E-n51-k5,5,14.76,TO (0.13),579.95,509.34,793.62,210.07,78.63
\end{filecontents*}
\csvreader[head to column names, late after line=\\]{tables/delta-fair-vrp.csv}{}
    {\csvcoli-\csvcolii & \dashToTO{\csvcoliii} & \dashToTO{\csvcolv} & \dashToTO{\csvcolvi} & \dashToTO{\csvcolvii} & \dashToTO{\csvcolviii} & \dashToTO{\csvcolix}}
    \bottomrule
\end{tabular}
\end{table}
\end{document}